\documentclass[12pt]{article}
\usepackage[centertags]{amsmath}
\usepackage{amsfonts}
\usepackage{amssymb}
\usepackage{amsthm}
\usepackage[top=3cm, bottom=3cm,left=3cm, right=3cm]{geometry}
\usepackage[all]{xy}

\begin{document}
 \newtheorem{thm}{Theorem}[section]
 \newtheorem{cor}[thm]{Corollary}
 \newtheorem{lem}[thm]{Lemma}
 \newtheorem{prop}[thm]{Proposition}
 \theoremstyle{definition}
 \newtheorem{defn}[thm]{Definition}
 \theoremstyle{remark}
 \newtheorem{rem}[thm]{Remark}
 \newtheorem{ex}[thm]{Example}
 \theoremstyle{example}
 \numberwithin{equation}{section}

 \title{Movability  and co-movability of shape morphisms}
 \author{P.S.Gevorgyan and I.Pop}
 \date{}
 \maketitle
\begin{abstract}The purpose of this paper is to define some notions of movability for morphisms of inverse systems which extend the movability properties of inverse systems and which are compatible with the equivalence relations which define pro-morphisms and shape morphisms. Some properties, examples and applications are given.

\textit{2000 Mathematics Subject Classification.}  55P55, 54C56,

\textit{Keywords.} Movability and co-movability (strong, uniform) for morphism of inverse systems, of pro-morphisms, of shape morphisms and maps; Mittag-Leffler property.
\end{abstract}

\section{Introduction}
The notion of movability for metric compacta was introduced by K.
Borsuk \cite{Borsuk} as an important shape invariant. The movable
spaces are a generalization of spaces having the shape of ANR's. The
movability assumption allows a series of important results in
algebraic topology (like the Whitehead and Hurewicz theorems) to
remain valid when the homotopy pro-groups are replaced by the
corresponding shape groups. The term "movability" comes from the
geometric interpretation of the definition in the compact case: if
$X$ is a compactum lying in a space $M\in AR$, one says that $X$ is
movable if for every neighborhood $U$ of $X$ in $M$ there exists a
neighborhood $V\subset U$ of $X$ such that for every neighborhood
$W\subset U$ of $X$ there is a homotopy $H:V\times [0,1]\rightarrow
U$ such that $H(x,0)=x $ and $H(x,1)\in W$ for every $x\in V$. One
shows that the choice of $M\in AR$ is irrelevant \cite{Borsuk}.
After the notion of movability had been expressed in terms of
ANR-systems, for arbitrary topological spaces, \cite{Mard-Segal1},
\cite{Mard-Segal2}, it became clear that one could define it in
arbitrary pro-categories. The definition of a movable object in an
arbitrary pro-category and that of uniform movability were both given by Maria
Moszy\'{n}ska \cite{Moszy}. Uniform movability is important in the
study of mono- and epi-morphisms in pro-categories and in the study
of the shape of pointed spaces. In the book of Sibe Marde\v{s}i\'{c} and
Jack Segal \cite{Mard-Segal2} all these approaches and applications
of various types of movability are discussed.

Some categorical approaches to movability in shape theory were given by P.S.Gevorgyan \cite{Gev}, I. Pop \cite{Pop}, P.S. Gevorgyan $\&$ I. Pop \cite{Gev-Pop}, T.A. Avakyan $\&$ P.S. Gevorgyan \cite{Ava-Gev}.

The approach of this paper started from the idea that some results in shape theory, valid in some conditions of movability for inverse systems or shape objects, could be obtained in weak conditions. More precisely, the idea was to define some movability notions for the morphisms of inverse systems, so as, when these morphisms are identity morphisms, the definitions of the usual movability properties to be found both for pro-objects and shape objects. Besides the classic case of movability pro-objects and shape objects, some notions of movability appear in the papers of T. Yagasaki about the study of the shape fibrations \cite{Yag1} and \cite{Yag2}, in a paper of Z.\v{C}erin on Lefschetz movable maps \cite{Cer}, and in a paper of D. A. Edwards and P. Tulley MacAuley about a shape theory of a map \cite{Edw}. Unfortunately, these approaches are just particular cases and they do not deal with  the movability of shape morphisms in the general case of an abstract shape theory.
 But the second author of this paper already approached this problem in two previous article \cite {Pop2} and \cite{Pop3} in which some properties of movability for morphisms of inverse systems and for pro-morphisms were defined and studied. But those properties of movability  could not be extended to shape morphisms and maps. In the present paper this problem is solved by a slight but useful modification of the notions of movability and uniform movability for morphisms of inverse systems previously proposed, the new definitions becoming compatible with the relations which define the pro-morphism and shape morphism. In this way we obtain some satisfactory notions of movability (simple, strong or uniform) for shape morphisms and particularly for maps. Some "dual" notions, namely that of co-movability (simple, strong or uniform), are defined. Some properties, examples and applications are given.

\section{Movable and strong movable morphisms of inverse systems and pro-morphisms }
All sets of indices are supposed to be cofinite directed sets. This condition is not restrictive (cf.\cite{Mard-Segal2}, Ch.I,\S1.2).

First we recall from \cite{Mard-Segal2},Ch. II, \S 6.1 the following definition.
\begin{defn}\label{def.2.1} An object $\mathbf{X}=(X_\lambda,p_{\lambda\lambda'},\Lambda)$ of pro-$\mathcal{C}$ is \textsl{movable} provided every $\lambda \in \Lambda$ admits a $\lambda'\geq \lambda$ (called a \textsl{movability index} of $\lambda$) such that each $\lambda''\geq \lambda$ admits a morphism $r:X_{\lambda'}\rightarrow X_{\lambda''}$ of $\mathcal{C}$ which satisfies
\begin{equation}
p_{\lambda\lambda''}\circ r=p_{\lambda\lambda'},
\end{equation}
i.e. , makes the following diagram commutative

$$
 \xymatrix{
X_{\lambda'} \ar[rr]^{p_{\lambda\lambda'}}\ar@{-->}[dr]_{r} & & X_{\lambda}\\
& X_{\lambda''} \ar[ur]_{p_{\lambda\lambda''}} &
 }
$$

\end{defn}
\begin{defn}\label{def.2.2}
Let $\mathbf{X}=(X_\lambda,p_{\lambda\lambda'},\Lambda)$ and $\mathbf{Y}=(Y_\mu,q_{\mu\mu'},M)$ inverse systems in a the category $\mathcal{C}$ and $(f_\mu,\phi):\mathbf{X}\rightarrow \mathbf{Y}$ a morphism of inverse systems.
We say that the morphism $(f_\mu,\phi)$ is \textsl{movable} if every $\mu\in M$ admits $\lambda\in \Lambda$, $\lambda\geq \phi(\mu)$,
such that each $\mu'\in M$, $\mu'\geq\mu$, admits a morphism $u:X_{\lambda}\rightarrow Y_{\mu'}$, in the category $\mathcal{C}$, which satisfies
\begin{equation}
f_\mu\circ p_{\phi(\mu)\lambda}=q_{\mu\mu'}\circ u
\end{equation}
i.e., makes the following diagram commutative

\[\xymatrix{X_{\phi(\mu)} \ar[r]^{f_\mu} & Y_\mu\\
X_\lambda\ar[u]^{p_{\phi(\mu)\lambda}}\ar@{-->} [r]_u & Y_{\mu'}\ar[u]_{q_{\mu\mu'}}}\]
The index $\lambda$ is called a \textsl{movability index} (MI) of $\mu$ with respect to the morphism $(f_\mu,\phi)$.

The composition $f_\mu\circ p_{\phi(\mu)\lambda}$, for $\lambda\geq \phi(\mu)$ is denoted by $f_{\mu\lambda}$ (cf.\cite{Mard-Segal2}, Ch.II,\S 2.1).
By this notation the relation (2.2) becomes
$$ f_{\mu\lambda}=q_{\mu\mu'}\circ u. $$

$$
 \xymatrix{
X_{\lambda} \ar[rr]^{f_{\mu\lambda}}\ar@{-->}[dr]_{u} & & Y_{\mu}\\
& Y_{\mu'} \ar[ur]_{q_{\mu\mu'}} &
 }
$$

\end{defn}
\begin{rem}\label{rem.2.3}If $\lambda$ is a movability index for $\mu$ with respect to $(f_\mu,\phi)$, then any $\widetilde{\lambda}>\lambda$
has this property.
\end{rem}
\begin{ex}\label{ex.2.4}
Let $(X)$ be a rudimentary system in the category $\mathcal{C}$ and  $(f_\mu):(X)\rightarrow \mathbf{Y}=(Y_\mu,q_{\mu\mu'},M)$ a morphism of inverse systems. Then $(f_\mu)$ is movable. Indeed, we have $f_\mu=q_{\mu\mu'}\circ f_{\mu'}$.

More generally, if we consider a morphism $(f_\mu,\phi):\mathbf{X}=(X_\lambda,p_{\lambda\lambda'},\Lambda)\rightarrow \mathbf{Y}=(Y_\mu,q_{\mu\mu'},M)$ such that there exists $\lambda_M\in \Lambda$ satisfying $\lambda_M\geq \phi(\mu)$ for any $\mu\in M$, then $(f_\mu,\phi)$ is movable. Indeed, for an arbitrary index $\mu\in M$ and $\mu'\geq \mu$ we have $f_\mu\circ p_{\phi(\mu)\lambda_M}=q_{\mu\mu'}\circ f_{\mu'}\circ p_{\phi(\mu')\lambda_M}=q_{\mu\mu'}\circ u$,
where $u=f_{\mu'}\circ p_{\phi(\mu')\lambda_M}$ is a morphism from $X_{\lambda_M}$ to $Y_{\mu'}$. So, $\lambda_M$ is a movability index.

\end{ex}

\begin{thm}\label{thm.2.5}
 An inverse system $\mathbf{X}=(X_\lambda,p_{\lambda\lambda'},\Lambda)$ is movable if and only if the identity morphism $1_{\mathbf{X}}$ is movable.
\end{thm}
\begin{proof}Let an inverse system $\mathbf{X}=(X_\lambda,p_{\lambda\lambda'},\Lambda)$ be movable. Consider an arbitrary $\lambda\in \Lambda$. Note
that a movability index $\lambda'\geq \lambda$ is also a movability index of $\lambda$ with respect to $1_{\mathbf{X}}:\mathbf{X}\rightarrow \mathbf{X}$.

Conversely, let $1_{\mathbf{X}}:\mathbf{X}\rightarrow \mathbf{X}$ be a movable morphism. Note that for any $\lambda\in \Lambda$ a movability index $\lambda'\geq \lambda$ of $\lambda$ with respect to $1_{\mathbf{X}}$ is also a movability index for the inverse system $\mathbf{X}=(X_\lambda,p_{\lambda\lambda'},\Lambda)$.
\end{proof}

\begin{thm}\label{thm.2.6}

Let $\mathbf{X}=(X_\lambda,p_{\lambda\lambda'},\Lambda)$, $\mathbf{Y}=(Y_\mu,q_{\mu\mu'},M)$, $\mathbf{Z}=(Z_\nu,r_{\nu\nu'},N)$ be inverse systems in the category $\mathcal{C}$ and $(f_\mu,\phi):\mathbf{X}\rightarrow \mathbf{Y}$, $(g_\nu,\psi):\mathbf{Y}\rightarrow \mathbf{Z}$ morphisms of inverse systems. Suppose that $(g_\nu,\psi)$ is movable. Then the composition $(h_\nu,\chi)=(g_\nu,\psi)\circ(f_\mu,\phi)$ is also a movable morphism.
\end{thm}
\begin{proof}
Recall that we have $\chi=\phi\circ \psi$ and $h_\nu=g_\nu\circ f_{\psi(\nu)}$. If $(g_\nu,\psi)$ is movable and $\nu\in N$, then for a movability index of $\nu\in N$,  $\mu\in M $, with  $\mu\geq \psi(\nu)$, an index $\nu'\in N$, with   $\nu'\geq \nu$ admits a morphism in $\mathcal{C}$, $u:Y_{\mu}\rightarrow Z_{\nu'}$ such that $ g_{\nu}\circ q_{\psi(\nu)\mu}=r_{\nu\nu'}\circ u$

\[\xymatrix{Y_{\psi(\nu)}\ar[r]^{g_{\nu}} & Z_\nu\\
Y_{\mu}\ar[u]^{q_{\psi(\nu)\mu}}\ar@{-->} [r]_u & Z_{\nu'}\ar[u]_{r_{\nu\nu'}}}\]

Now consider  $\lambda\in \Lambda$ with $\lambda\geq \phi(\mu)$, $\lambda\geq \phi(\psi(\nu))$ such that $f_{\psi(\nu)}\circ p_{\phi(\psi(\nu))\lambda}=q_{\psi(\nu)\mu}\circ f_{\mu}\circ p_{\phi(\mu)\lambda}$.
Then we can consider the morphism $u':=u\circ f_{\mu}\circ p_{\phi(\mu)\lambda}:X_{\lambda}\rightarrow Z_{\nu'}$. For this morphism we obtain: $r_{\nu\nu'}\circ u'= (r_{\nu\nu'}\circ u)\circ f_{\mu}\circ p_{\phi(\mu)\lambda}= g_\nu\circ q_{\psi(\nu)}\circ f_{\mu}\circ p_{\phi(\mu)\lambda}=
g_\nu\circ f_{\psi(\nu)}\circ p_{\phi(\psi(\nu))}=h_\nu\circ p_{\chi(\nu)\lambda}$.

\[\xymatrix{X_{\chi(\nu)} \ar[r]^{h_\nu} & Z_\nu\\
X_{\lambda}\ar[u]^{p_{\chi(\nu)\lambda}}\ar@{-->} [r]_{u'} & Z_{\nu'}\ar[u]_{r_{\nu\nu'}}}\]

\end{proof}
\begin{cor}\label{cor.2.7}Let an inverse system $\mathbf{X}=(X_\lambda,p_{\lambda\lambda'},\Lambda)$ be arbitrary and $\mathbf{Y}=(Y_\mu,q_{\mu\mu'},M)$
be movable. Then any morphism $(f_\mu,\phi):\mathbf{X}\rightarrow \mathbf{Y}$ is movable.
\end{cor}
\begin{proof}Since $(f_\mu,\phi)=1_{\mathbf{Y}}\circ (f_\mu,\phi)$ and $1_{\mathbf{Y}}:\mathbf{Y}\rightarrow \mathbf{Y}$ is a movable morphism by Theorem
\ref{thm.2.5}, then $(f_\mu,\phi)$ is also movable according to Theorem \ref{thm.2.6}.
\end{proof}

\begin{thm}\label{thm.2.8}
Let $\mathbf{X}=(X_\lambda,p_{\lambda\lambda'},\Lambda)$, $\mathbf{Y}=(Y_\mu,q_{\mu\mu'},M)$, $\mathbf{Z}=(Z_\nu,r_{\nu\nu'},N)$ be inverse systems in the category $\mathcal{C}$ and $(f_\mu,\phi):\mathbf{X}\rightarrow \mathbf{Y}$, $(g_\nu,\psi):\mathbf{Y}\rightarrow \mathbf{Z}$ morphisms of inverse systems. Suppose that $(f_\mu,\phi)$ is movable. Then the composition $(h_\nu,\chi)=(g_\nu,\psi)\circ(f_\mu,\phi)$ is also a movable morphism.
\end{thm}
\begin{proof}For an arbitrary index $\nu\in N$, consider a movability index  $\lambda$ of $\psi(\nu)$, $\lambda\geq \phi(\psi(\nu))$, with respect to
$(f_\mu,\phi)$. Let us prove that $\lambda$ is a movability index of $\nu$ with respect to the morphism $(h_\nu,\chi)$.

Let $\nu'\in N,\nu'\geq \nu$ be any index.  Consider $\mu'\geq \psi(\nu'),\psi(\nu)$, and such that
$$ r_{\nu\nu'}\circ g_{\nu'}\circ q_{\psi(\nu')\mu'}=g_\nu\circ q_{\psi(\nu)\mu'}.$$

For $q_{\psi(\nu)\mu'}$ by the movability of $(f_\mu,\phi):\mathbf{X}\rightarrow \mathbf{Y}$ there exists a morphism $u:X_{\lambda}\rightarrow Y_{\mu'}$ such that
$$  f_{\psi(\nu)}\circ p_{\phi(\psi(\nu))\lambda}=q_{\psi(\nu)\mu'}\circ u,$$
and define $U:X_{\lambda}\rightarrow Z_{\nu'}$, by
$$ U=g_{\nu'}\circ q_{\psi(\nu')\mu'}\circ u. $$

Now we have: $r_{\nu\nu'}\circ U=r_{\nu\nu'}\circ g_{\nu'}\circ q_{\psi(\nu')\mu'}\circ u=g_\nu\circ q_{\psi(\nu)\mu'}\circ u=
g_\nu\circ q_{\psi(\nu)\mu}\circ f_{\mu}\circ p_{\phi(\mu)\lambda}=g_\nu\circ f_{\psi(\nu)}\circ p_{\chi(\nu)\lambda}=h_\nu\circ p_{\chi(\nu)\lambda}$.

\end{proof}
\begin{cor}\label{cor.2.9}Let $\mathbf{X}=(X_\lambda,p_{\lambda\lambda'},\Lambda)$ be movable and $\mathbf{Y}=(Y_\mu,q_{\mu\mu'},M)$ be arbitrary
inverse system. Then any morphism $(f_\mu,\phi):\mathbf{X}\rightarrow \mathbf{Y}$ is movable.
\end{cor}
\begin{proof}Since $(f_\mu,\phi)=(f_\mu,\phi)\circ 1_{\mathbf{X}}$ and $1_{\mathbf{X}}:\mathbf{X}\rightarrow \mathbf{X}$ is a movable morphism by Theorem
\ref{thm.2.5}, then $(f_\mu,\phi)$ is also movable according to Theorem \ref{thm.2.8}.
\end{proof}

\begin{prop}\label{prop.2.10}
Let $(f_\mu,\phi),(f'_\mu,\phi'):\mathbf{X}=(X_\lambda,p_{\lambda\lambda'},\Lambda)\rightarrow \mathbf{Y}=(Y_\mu,q_{\mu\mu'},M)$ be two equivalent morphisms of inverse systems, $(f_\mu,\phi)\sim (f'_\mu,\phi')$. If the morphism $(f_\mu,\phi)$ is movable then the morphism $(f'_\mu,\phi')$ is also movable.
\end{prop}
\begin{proof}
Let $\mu\in M$ with $\lambda$ a movability index  for $\mu$ with respect to the morphism $(f_\mu,\phi)$. Now by the definition of the equivalence relation between morphisms of inverse systems (see \cite{Mard-Segal2},Ch.I,\S1.1), for the index $\mu$ there exist an index $\lambda'\geq \phi(\mu),\phi'(\mu)$ such that
\begin{equation}
 f_\mu\circ p_{\phi(\mu)\lambda'}=f'_\mu\circ p_{\phi'(\mu)\lambda'},
\end{equation}
and we can suppose that $\lambda'\geq \lambda$, such that we refer at $\lambda'$ as a movability index of $\mu$ with respect to $(f_\mu,\phi)$. Then we can prove that $\lambda'$ is also a movability index for $\mu$ with respect to the morphism
$(f'_\mu,\phi')$. Let be $\mu'\in M,\mu' \geq\mu$. Then there exists a morphism $u:X_{\lambda'}\rightarrow Y_{\mu'}$ such that $ f_{\mu\lambda'} =q_{\mu\mu'}\circ u$,

$$
 \xymatrix{
X_{\lambda'} \ar[rr]^{f_{\mu\lambda'}}\ar@{-->}[dr]_{u} & & Y_\mu\\
& Y_{\mu'} \ar[ur]_{q_{\mu\mu'}} &
 }
$$

But by relation (2.3) we have that $f_{\mu\lambda'}=f'_{\mu\lambda'}$ such that $f'_{\mu\lambda'}=q_{\mu\mu'}\circ u$.

\end{proof}
Recall that the pro-category pro-$\mathcal{C}$ of a category $\mathcal{C}$ has as objects all inverse systems $\mathbf{X}$ in $\mathcal{C}$ (over all directed sets $\Lambda$) and as morphisms $\mathbf{f:X\rightarrow Y}$ the equivalence classes of morphisms of inverse systems
$(f_\mu,\phi):\mathbf{X}\rightarrow \mathbf{Y}$ with respect to the equivalence relation $\sim$ above considered.
Thanks to Proposition \ref{prop.2.10} we can give the following definition.
\begin{defn}\label{defn.2.11}
A morphism in a pro-category pro-$\mathcal{C}$, $\mathbf{f}:\mathbf{X}\rightarrow \mathbf{Y}$,  is called \textsl{movable} if $\mathbf{f}$ admits a representation $(f_\mu,\phi):\mathbf{X}\rightarrow \mathbf{Y}$ which is movable.
\end{defn}
The next theorem follows from Theorems \ref{thm.2.6}, \ref{thm.2.8} and Corollaries \ref{cor.2.7}, \ref{cor.2.9}.
\begin{thm}\label{thm.2.12}
A composition of an arbitrary pro-morphism  at right or left with a movable pro-morphism is a movable pro-morphism. Particularly, if for a pro-morphism $\mathbf{f}:\mathbf{X}\rightarrow \mathbf{Y}$, one of $\mathbf{X}$ or $\mathbf{Y}$ is a movable system, then $\mathbf{f}$ is a movable pro-morphism.
\end{thm}

\begin{cor}\label{cor.2.13}Let $\mathbf{X}=(X_\lambda,p_{\lambda\lambda'},\Lambda)$ and $\mathbf{Y}=(Y_\mu,q_{\mu\mu'},M)$ in pro-$\mathcal{C}$. If $\mathbf{Y}$ is movable and $\mathbf{X}$ is dominated by $\mathbf{Y}$ in pro-$\mathcal{C}$, $\mathbf{X}\leq \mathbf{Y}$,  then $\mathbf{X}$ is also movable.
\end{cor}
\begin{proof}By assumption there exist the pro-morphisms $\mathbf{f}:\mathbf{X}\rightarrow \mathbf{Y}$ and $\mathbf{g}:\mathbf{Y}\rightarrow \mathbf{X}$
such that $\mathbf{g}\circ \mathbf{f}=1_{\mathbf{X}}$. Because $\mathbf{Y}$ is movable, by Theorem  \ref{thm.2.12} it follows that $\mathbf{f}$, $\mathbf{g}$ and $1_\mathbf{X}$ are movable pro-morphisms and hence, $\mathbf{X}$ is movable by Theorem \ref{thm.2.5}.
\end{proof}
\begin{rem}\label{rem.2.14}
Corollary \ref{cor.2.13} is a new proof of an important result in the theory of shape (see \cite{Mard-Segal2} Theorem 1, Ch. II, \S 6.1).
\end{rem}
\begin{defn}\label{def.2.15}
A morphism of inverse systems $(f_{\mu},\phi):((X_\lambda,*),p_{\lambda\lambda'},\Lambda)\rightarrow ((Y_\mu,*),q_{\mu\mu'},M)$ of pointed sets is said to have the \textsl{Mittag-Leffler property} provided every $\mu\in M$ admits a $\lambda\in \Lambda, \lambda \geq \phi(\mu)$ (an ML \textsl{index for $\mu$ with respect to} $(f_\mu,\phi)$, such that for any $\lambda'\in \Lambda, \lambda'\geq \lambda$   one has

\begin{equation}
 f_{\mu\lambda'}(X_{\lambda'})=f_{\mu\lambda}(X_\lambda).
\end{equation}

If $(f_\mu,\phi)$ is replaced by $1_{(\mathbf{X},*)}$ we obtain the Mittag-Leffler property for an inverse system in the category pro-Set$_*$ (cf.\cite{Mard-Segal2},Ch. II, \S 6.2).

The same condition defines the Mittag-Leffler property for a morphism of pro-groups.
\end{defn}
\begin{prop}\label{prop.2.16}A morphism of inverse systems of pointed sets with the Mittag-Leffler property is movable.
\end{prop}
\begin{proof}Let $(f_\mu,\phi):((X_\lambda,*),p_{\lambda\lambda'},\Lambda)\rightarrow ((Y_\mu,*),q_{\mu\mu'},M)$ with the Mittag-Leffler property.
Then for $\mu\in M$ there is an ML index $\lambda \in \Lambda, \lambda\geq \phi(\mu)$  such that (2.4) holds for each $\lambda'\geq \lambda$. We can prove that $\lambda$ is a movability index for $\mu$ with respect to $(f_\mu,\phi)$. For $\mu'\in M$, $\mu'\geq \mu$, consider $\lambda'\in \Lambda$ satisfying the conditions $\lambda'\geq \lambda\geq\phi(\mu)$, $\lambda'\geq \phi(\mu')$ and
$$ q_{\mu\mu'}\circ f_{\mu'}\circ p_{\phi(\mu')\lambda'}=f_\mu\circ p_{\phi(\mu)\lambda'},$$
i.e.,
\begin{equation}
q_{\mu\mu'}\circ f_{\mu'\lambda'}=f_{\mu\lambda'}.
\end{equation}
Then since $\lambda'\geq \lambda$ we have also satisfied (2.4). This last relation implies the existence of a function $r:(X_\lambda,*)\rightarrow
(X_{\lambda'},*)$ satisfying
\begin{equation}
f_{\mu\lambda'}\circ r=f_{\mu\lambda}.
\end{equation}
Now from the relations (2.5), (2.6) we have: $f_{\mu\lambda}=q_{\mu\mu'}\circ f_{\mu'\lambda'}\circ r$, so that we denote $u:=f_{\mu'\lambda'}\circ r:
(X_\lambda,*)\rightarrow (Y_{\mu'},*)$,
we have
$$ f_{\mu\lambda}=q_{\mu\mu'}\circ u,$$
which is the movability condition.

\end{proof}

\begin{prop}\label{prop.2.17}
Consider in the category Grp of groups an inverse system $\mathbf{G}=(G_\lambda,q_{\mu\mu'},M)$ and in the category $\mathcal{F}$ of free groups an inverse systems $\mathbf{F}=(F_\lambda,p_{\lambda\lambda'},\Lambda)$. Then if a morphism $(f_\mu,\phi):\mathbf{F}\rightarrow \mathbf{G}$  is  movable as a morphism in pro-Set$_\ast$, it is movable in pro-Grp.
\end{prop}
\begin{proof}
Let for $\mu\in M $ a movability index $\lambda$ with respect to $(f_\mu,\phi)$ in pro-Set$_\ast$. If $\mu'\geq \mu$, then there exists a function $r:F_{\lambda}\rightarrow G_{\mu'}$ satisfying $f_{\mu\lambda}=q_{\mu\mu'}\circ r$. If $B$ is a basis for $F_{\lambda}$ then there exists a homomorphism $u:F_{\lambda}\rightarrow G_{\mu'}$ such that $u|B=r|B$. Then the homomorphisms $ f_{\mu\lambda}$ and $q_{\mu\mu'}\circ u$ coincide on $B$ and therefore they are equal.
\end{proof}
For a new property of movability for morphisms of inverse systems we recall from \cite{Mard-Segal2},Ch. II, \S 9.3 the following definition.
\begin{defn}\label{def.2.18}
An object $\mathbf{X}=(X_\lambda,p_{\lambda\lambda'},\Lambda)\in $ pro-$\mathcal{C}$ is called \textsl{strongly movable} provided every $\lambda\in \Lambda$ admits a $\lambda'\geq \lambda$ (a \textsl{strong movability index} of $\lambda$) such that for every $\lambda''\geq \lambda$ there exist a $\lambda^\ast\geq \lambda',\lambda''$ and a morphism $r:X_{\lambda'}\rightarrow X_{\lambda''}$ satisfying
\begin{equation}
p_{\lambda\lambda''}\circ r=p_{\lambda\lambda'},
\end{equation}
\begin{equation}
r\circ p_{\lambda'\lambda^\ast}=p_{\lambda''\lambda^\ast},
\end{equation}
i.e., rendering the following diagram commutative.
$$
 \xymatrix{
{X_{\lambda^\ast}}\ar[rr]^{p_{\lambda'\lambda^\ast}}\ar[dd]_{p_{\lambda''\lambda^\ast}} && {X_{\lambda'}}\ar@{-->}[ddll]_{r}\ar[dd]^{p_{\lambda\lambda'}}\\
& &\\
{X_{\lambda''}}\ar[rr]_{p_{\lambda\lambda''}} && {X_{\lambda}}
 }
$$
\end{defn}
Obviously that strong movability implies movability of an inverse system.
\begin{defn}\label{def.2.19}
Let $\mathbf{X}=(X_\lambda,p_{\lambda\lambda'},\Lambda)$ and $\mathbf{Y}=(Y_\mu,q_{\mu\mu'},M)$ inverse systems in a the category $\mathcal{C}$ and $(f_\mu,\phi):\mathbf{X}\rightarrow \mathbf{Y}$ a morphism of inverse systems.
We say that the morphism $(f_\mu,\phi)$ is \textsl{strongly} \textsl{movable} if every $\mu\in M$ admits($\lambda\geq \phi(\mu)$), a \textsl{strong movability index} of $\mu$ with respect to $(f_\mu,\phi)$ such that every $\mu'\geq \mu$ exists  a $\lambda^\ast\geq \lambda,\phi(\mu')$ and a morphism $u:X_\lambda\rightarrow Y_{\mu'}$ satisfying
\begin{equation}
q_{\mu\mu'}\circ u= f_{\mu\lambda},
\end{equation}
\begin{equation}
u\circ p_{\lambda\lambda^\ast}=f_{\mu'\lambda^\ast},
\end{equation}
i.e., the following diagram is commutative.

$$
 \xymatrix{
X_{\lambda^\ast}\ar[rr]^{p_{\lambda\lambda^\ast}}\ar[dd]_{f_{\mu'\lambda^\ast}} && X_\lambda\ar@{-->}[ddll]_u\ar[dd]^{f_{\mu\lambda}}\\
& &\\
Y_{\mu'}\ar[rr]_{q_{\mu\mu'}} && Y_\mu
 }
$$

\end{defn}
Obviously that strong movability of morphisms implies movability.

\begin{rem}\label{rem.2.20}
An inverse system $\mathbf{X}=(X_\lambda,p_{\lambda\lambda'},\Lambda)$ is strongly movable if and only if the identity morphism $1_{\mathbf{X}}$ is strongly movable.
\end{rem}
\begin{prop}\label{prop.2.21}

Let $\mathbf{X}=(X_\lambda,p_{\lambda\lambda'},\Lambda)$, $\mathbf{Y}=(Y_\mu,q_{\mu\mu'},M)$, $\mathbf{Z}=(Z_\nu,r_{\nu\nu'},N)$ be inverse systems in the category $\mathcal{C}$ and $(f_\mu,\phi):\mathbf{X}\rightarrow \mathbf{Y}$, $(g_\nu,\psi):\mathbf{Y}\rightarrow \mathbf{Z}$ morphisms of inverse systems. Suppose that $(g_\nu,\psi)$ is strongly movable. Then the composition $(h_\nu,\chi)=(g_\nu,\psi)\circ(f_\mu,\phi)$ is also a strongly movable morphism.
\end{prop}
\begin{proof}
We use the notations and the results from the proof of Theorem \ref{thm.2.6}. Suppose that for the morphism $u:Y_{\mu'}\rightarrow Z_{\nu'}$ there exists $\mu^\ast\geq \mu,\psi(\nu')$ such that we have $u\circ q_{\mu\mu^\ast}=g_{\nu'\mu^\ast}$. Consider $\lambda^\ast\geq \lambda,\phi(\mu^\ast), \chi(\nu)=\phi(\psi(\nu))$. Then for the morphism $u'=u\circ f_{\mu}\circ p_{\phi(\mu)\lambda}: X_{\lambda'}\rightarrow Z_{\nu'}$ we have: $u'\circ p_{\lambda\lambda^\ast}=u\circ f_{\mu}\circ p_{\phi(\mu)\lambda}\circ p_{\lambda\lambda^\ast}=u\circ f_{\mu}\circ p_{\phi(\mu)\lambda^\ast}=u\circ q_{\mu'\mu^\ast}\circ f_{\mu^\ast}\circ p_{\phi(\mu^\ast)\lambda^\ast}=g_{\nu''\mu^\ast}\circ f_{\mu^\ast}\circ p_{\phi(\mu^\ast)\lambda^\ast}=g_{\nu'}\circ q_{\psi(\nu')\mu^\ast}\circ f_{\mu^\ast}\circ p_{\phi(\mu^\ast)\lambda^\ast}=g_{\nu'}\circ f_{\psi(\nu')}\circ p_{\phi(\psi(\nu'))\lambda^\ast}=h_{\nu'\lambda^\ast}$.
\end{proof}

\begin{prop}\label{prop.2.22}

Let $\mathbf{X}=(X_\lambda,p_{\lambda\lambda'},\Lambda)$, $\mathbf{Y}=(Y_\mu,q_{\mu\mu'},M)$, $\mathbf{Z}=(Z_\nu,r_{\nu\nu'},N)$ be inverse systems in the category $\mathcal{C}$ and $(f_\mu,\phi):\mathbf{X}\rightarrow \mathbf{Y}$ , $(g_\nu,\psi):\mathbf{Y}\rightarrow \mathbf{Z}$ morphisms of inverse systems. Suppose that $(f_\mu,\phi)$ is strongly movable. Then the composition $(h_\nu,\chi)=(g_\nu,\psi)\circ(f_\mu,\phi)$ is also a strongly movable morphism.

Particularly, if $\mathbf{X}=(X_\lambda,p_{\lambda\lambda'},\Lambda)$ is a strongly movable system, then any morphism $(f_\mu,\phi):\mathbf{X}\rightarrow
 \mathbf{Y} =(Y_\mu,q_{\mu\mu'},M)$ is a strongly movable morphism.
\end{prop}
\begin{proof}
We use the notations and the results from the proof of Theorem \ref{thm.2.8}.Suppose that for the morphism $u:X_{\lambda}\rightarrow Y_{\mu'}$ there exists $\lambda^\ast \geq\lambda,\phi(\mu'),\phi(\psi(\nu'))$,  such that $u\circ p_{\lambda\lambda^\ast}=f_{\mu'\lambda^\ast}$. Then for the morphism $U:X_{\lambda}\rightarrow Z_{\nu'}$ we have $U\circ p_{\lambda\lambda^\ast}=g_{\nu'}\circ q_{\psi(\nu')\mu'}\circ u\circ p_{\lambda\lambda^\ast}=g_{\nu'}\circ q_{\psi(\nu')\mu'}\circ f_{\mu'}\circ p_{\phi(\mu')\lambda^\ast}=g_{\nu'}\circ f_{\psi(\nu')}
\circ p_{\chi(\nu')\lambda^\ast}=h_{\nu'\lambda^\ast}$.
\end{proof}
\begin{prop}\label{prop.2.23}
Let $(f_\mu,\phi),(f'_\mu,\phi'):\mathbf{X}=(X_\lambda,p_{\lambda\lambda'},\Lambda)\rightarrow \mathbf{Y}=(Y_\mu,q_{\mu\mu'},M)$ be two equivalent morphisms of inverse systems, $(f_\mu,\phi)\sim (f'_\mu,\phi')$. If the morphism $(f_\mu,\phi)$ is strongly movable then the morphism $(f'_\mu,\phi')$ is also strongly movable.
\end{prop}
\begin{proof}
We use the notations and the results from the proof of Proposition \ref{prop.2.10}. Suppose that for the morphism $u:X_{\lambda'}\rightarrow Y_{\mu'}$ there exists $\lambda^\ast\geq \lambda,\phi(\mu')$ such that $u\circ p_{\lambda\lambda^\ast} =f_{\mu'\lambda^\ast}$. Consider $\overline{\lambda^\ast}
\geq \phi(\mu'),\phi'(\mu'),\lambda^\ast$ such that $f_\mu'\circ p_{\phi(\mu')\overline{\lambda^\ast}}=f'_{\mu'}\circ p_{\phi'(\mu')\overline{\lambda^\ast}}$. Then we have
$u\circ p_{\lambda\overline{\lambda^\ast}} =f'_{\mu'\overline{\lambda^\ast}}$.
\end{proof}
Thanks to Proposition \ref{prop.2.23} we can give the following definition.
\begin{defn}\label{defn.2.24}
A morphism in a pro-category pro-$\mathcal{C}$, $\mathbf{f}:\mathbf{X}\rightarrow \mathbf{Y}$,  is called \textsl{strongly } \textsl{movable} if $\mathbf{f}$ admits a representation $(f_\mu,\phi):\mathbf{X}\rightarrow \mathbf{Y}$ which is strongly movable.
\end{defn}

\begin{rem}\label{rem.2.25}
Using Propositions \ref{prop.2.21} and \ref{prop.2.22} one can prove Corollaries \ref{cor.2.7}, \ref{cor.2.9} and \ref{cor.2.13} replacing the property of movability by the property of strong movability for pro-morphisms.
\end{rem}
Now for to give an important class of movable morphisms we recall the following definition(\cite{Mard-Segal2},Ch.II, \S 6.1, p.104). Let $\mathcal{C}_0$ be a full subcategory of the category $\mathcal{C}$. An inverse system $\mathbf{X}=(X_\lambda,p_{\lambda\lambda'},\Lambda)\in pro-\mathcal{C}$ is $\mathcal{C}_0$-\textsl{movable} provided each $\lambda \in \Lambda$ admits a $\lambda'\geq \lambda$ such that for any $\lambda''\geq \lambda$, for any object $X_0\in \mathcal{C}_0$, and any morphism $h:X_0\rightarrow X_{\lambda'}$ in $\mathcal{C}$ there exists a morphism $r:X_0\rightarrow X_{\lambda''}$ in $\mathcal{C}$ such that
$$ p_{\lambda\lambda''}\circ r=p_{\lambda\lambda'}\circ h.$$
It follows immediately that the movability in $pro-\mathcal{C}$ implies $\mathcal{C}_0$-movability for every $\mathcal{C}_0$. The following proposition generalizes Corollary \ref{cor.2.7}.
\begin{prop}\label{prop.2.26}Let $\mathcal{C}_0$ be a full subcategory of the category $\mathcal{C}$ and let $\mathbf{Y}=(Y_\mu,q_{\mu\mu'},M)\in pro- \mathcal{C}$ be a $\mathcal{C}_0$-movable system. Then if $\mathbf{X}=(X_\lambda,p_{\lambda\lambda'},\Lambda)\in pro-\mathcal{C}_0$, any morphism of inverse systems $(f_\mu,\phi):\mathbf{X}\rightarrow \mathbf{Y}$ is movable.
\end{prop}
\begin{proof}For $\mu\in M$ let $\mu'$ be as in the above definition. Then for $X_{\phi(\mu')}\in \mathcal{C}_0$ and $f_{\mu'}:X_{\phi(\mu')}\rightarrow Y_{\mu'}$ and $q''\geq q$, in $M$, there exists $r:X_{\phi(\mu')}\rightarrow Y_{\mu''}$ such that
\begin{equation}
q_{\mu\mu''}\circ r=q_{\mu\mu'}\circ f_{\mu'} .
\end{equation}

Consider $\lambda\in \Lambda, \lambda\geq \phi(\mu),\phi(\mu')$ such that
\begin{equation}
 q_{\mu\mu'}\circ f_{\mu'}\circ p_{\phi(\mu')\lambda}=f_{\mu}\circ p_{\phi(\mu)\lambda}.
\end{equation}
Then we can verify that $\lambda$ is a movability index of $\mu$ with respect to the morphism $(f_\mu,\phi)$. Indeed if we compose at right the relation (2.11) by $p_{\phi(\mu')\lambda}$, by (2.12) we obtain $f_\mu\circ p_{\phi(\mu)\lambda}=q_{\mu\mu''}\circ r\circ p_{\phi(\mu')\lambda})$, i.e.,
$f_{\mu\lambda}=q_{\mu\mu''}\circ u$, for $u:=r\circ p_{\phi(\mu')\lambda}:X_{\lambda}\rightarrow Y_{\mu''}$.
\end{proof}
\begin{ex}\label{ex.2.27} Let $\mathbf{G}=(G_n,q_{nn+1})$ be an inverse sequence of groups defined as follows (\cite{Mard-Segal2} , Ch.II, \S 6.2, p. 166). Let $G_n=\mathbb{Z}_{2^n},n=1,2,...$, and $q_{nn+1}$ send the generator $[1]$ of $\mathbb{Z}_{2^{n+1}}$ to the generator $[1]$ of $\mathbb{Z}_{2^n}$. This system is not movable (see \cite{Mard-Segal2}, p.166), but since $q_{nn+1}$ are epimorphisms, $\mathbf{G}$ has the Mittag-Leffler property, and by Corollary 5 from \cite{Mard-Segal2}, Ch.II, \S 6.2,this system is movable with respect to free groups. Then by Proposition \ref{prop.2.26} it follows that any morphism $(f_n,\phi):\mathbf{F}=(F_\lambda,p_{\lambda\lambda'},\Lambda)\rightarrow \mathbf{G}$,  with $\mathbf{F}$ an inverse system of free groups, is a movable morphism.

More specifically. Consider the inverse sequence $\mathbf{F}=(F_n,p_{nn+1})$, with $F_n=\mathbb{Z}$ and $p_{nn+1}$ multiplication by 2. This is also a non-movable pro-group (cf.\cite{Mard-Segal2}, Ch.II, \S 6.1, p. 159). Denote  by $f_n=\pi_{2^n}:F_n=\mathbb{Z}\rightarrow G_n=\mathbb{Z}_{2^n}$ the canonical surjective group homomorphisms. The the sequence $(f_n)$ is a morphism of inverse systems. Indeed, for a morphism $g_{n,n+k}:G_{n+k}\rightarrow G_n$, if we take $m\geq 2n,n+k$, then $q_{nn+k}\circ f_{n+k}\circ p_{n+km}=f_n\circ p_{nm}$. And by the above this is a movable morphism. 
\end{ex}
\section{Uniformly movable pro-morphisms}
First we recall from \cite{Mard-Segal2}, Ch. II, \S 6.1 the following definition.
\begin{defn}\label{defn.3.1}
An object $\textbf{X}=(X_\lambda,p_{\lambda\lambda'},\Lambda)\in$ pro-$\mathcal{C}$ is called \textsl{uniformly movable} if every $\lambda\in \Lambda$ admits a $\lambda'\geq \lambda$(a \textsl{uniform movability index} of $\lambda$) such that there is a morphism $\mathbf{r}:X_{\lambda'}\rightarrow \mathbf{X}$ in pro-$\mathcal{C}$ satisfying
\begin{equation}
p_{\lambda\lambda'}=\mathbf{p}_\lambda \circ \mathbf{r}
\end{equation}
where $\mathbf{p}_\lambda:\mathbf{X}\rightarrow X_\lambda$ is the morphism of pro-$\mathcal{C}$ given by $1_\lambda:X_\lambda\rightarrow X_\lambda$, i.e., $\mathbf{p}_\lambda$ is the restriction of $\mathbf{X}$ to $X_\lambda$. Consequently, $p_{\lambda\lambda'}$ factors through $\mathbf{X}$.
\end{defn}
\begin{defn}\label{defn.3.2}
Let $\mathbf{X}=(X_\lambda,p_{\lambda\lambda'},\Lambda)$ and $\mathbf{Y}=(Y_\mu,q_{\mu\mu'},M)$ be inverse systems in a category $\mathcal{C}$ and $(f_\mu,\phi):\mathbf{X}\rightarrow \mathbf{Y}$ a morphism of inverse systems.
We say that the morphism $(f_\mu,\phi)$ is \textsl{uniformly} \textsl{movable} if every $\mu\in M$ admits $\lambda\in \Lambda, \lambda\geq
 \phi(\mu)$  (called a \textsl{uniform} \textsl{movability index} \textsl{of} $\mu$ \textsl{with respect to} $(f_\mu,\phi)$)such that there is a morphism of inverse systems $\mathbf{u}:X_{\lambda}\rightarrow \mathbf{Y}$ satisfying
\begin{equation}
 f_{\mu\lambda}=\mathbf{q}_\mu \circ \mathbf{u}
\end{equation}
i.e., the following diagram commutes

$$
 \xymatrix{
X_{\lambda} \ar[rr]^{f_{\mu\lambda}}\ar@{-->}[dr]_{\mathbf{u}} & & Y_{\mu}\\
& \mathbf{Y} \ar[ur]_{\mathbf{q}_{\mu}} &
 }
$$

where $\mathbf{q}_{\mu}:\mathbf{Y}\rightarrow Y_\mu$ is the morphism of inverse systems given by $1_\mu:Y_\mu\rightarrow Y_\mu$.
\end{defn}
\begin{rem}\label{rem.3.3}
If $\lambda$ is a uniform movability index, then any $\widetilde{\lambda}>\lambda$ has this property.

\end{rem}

\begin{rem}\label{rem.3.4}
Note that the morphism $\mathbf{u}:X_{\lambda}\rightarrow \mathbf{Y}$ determines for every $\mu_1\in M$ a morphism $u_{\mu_1}:X_{\lambda}\rightarrow Y_{\mu_1}$ in $\mathcal{C}$ such that for $\mu_1\leq \mu_2$ we have $q_{\mu_1\mu_2}\circ u_{\mu_2}=u_{\mu_1}$ and $u_\mu=f_{\mu\lambda}$. In particular, for $\mu'\in M$, $\mu'\geq \mu$, we have $q_{\mu\mu'}\circ u_{\mu'}=u_{\mu}=f_{\mu\lambda}$, so that uniform movability of morphisms implies movability.

\end{rem}
\begin{ex}\label{ex.3.5}
If $(X)$ is a rudimentary system in the category $\mathcal{C}$ then any morphism of inverse systems $(f_\mu):(X)\rightarrow \mathbf{Y}=(Y_\mu,q_{\mu\mu'},M)$ in the category $\mathcal{C}$ is uniformly movable because $ f_\mu\circ 1_{X}=\mathbf{q}_\mu\circ (f_\mu)$.More generally, if we consider a morphism $(f_\mu,\phi):\mathbf{X}=(X_\lambda,p_{\lambda\lambda'},\Lambda)\rightarrow \mathbf{Y}=(Y_\mu,q_{\mu\mu'},M)$ such that there exists $\lambda_M\in \Lambda$ satisfying $\lambda_M\geq \phi(\mu)$ for any $\mu\in M$, then $(f_\mu,\phi)$ is uniformly movable. Indeed, for an arbitrary index $\mu\in M$, a uniformly movable index  is $\lambda_M$, with $ f_{\mu\lambda_M}=\mathbf{q}_{\mu}\circ \mathbf{u}$, where $\mathbf{u}:X_{\lambda_M}\rightarrow \mathbf{Y}$ is defined by the components $u_{\mu'}:X_{\lambda_M}\rightarrow Y_{\mu'}$ given by $u_{\mu'}=
f_{\mu'\lambda_M}$.
\end{ex}
\begin{thm}\label{thm.3.6} An inverse system $\mathbf{X}=(X_\lambda,p_{\lambda\lambda'},\Lambda)$ is uniformly movable if and only if the identity morphism $1_{\mathbf{X}}$ is uniformly movable.
\end{thm}
\begin{proof}Let an inverse $\mathbf{X}=(X_\lambda,p_{\lambda\lambda'},\Lambda)$ be uniformly movable. Consider an arbitrary $\lambda\in\Lambda$. Note that a uniform movability index $\lambda'\geq \lambda$ is also a uniform movability index of $\lambda$ with respect to $1_\mathbf{X}:\mathbf{X}\rightarrow \mathbf{X}$.

Conversely, let $1_\mathbf{X}:\mathbf{X}\rightarrow \mathbf{X}$ be uniformly movable morphism. Note that for any $\lambda\in \Lambda$ a uniform movability index $\lambda'\geq \lambda $ of $\lambda$ with respect to $1_\mathbf{X}$ is also a uniform movability index for the inverse system $\mathbf{X}=(X_\lambda,p_{\lambda\lambda'},\Lambda)$.
\end{proof}
\begin{thm}\label{thm.3.7}
Let $\mathbf{X}=(X_\lambda,p_{\lambda\lambda'},\Lambda)$, $\mathbf{Y}=(Y_\mu,q_{\mu\mu'},M)$, $\mathbf{Z}=(Z_\nu,r_{\nu\nu'},N)$ be inverse systems in the category $\mathcal{C}$ and $(f_\mu,\phi):\mathbf{X}\rightarrow \mathbf{Y}$, $(g_\nu,\psi):\mathbf{Y}\rightarrow \mathbf{Z}$ morphisms of inverse systems. Then, if $(g_\nu,\psi)$ is uniformly movable, the composition $(h_\nu,\chi)=(g_\nu,\psi)\circ(f_\mu,\phi)$ is also a uniformly movable morphism.
\end{thm}
\begin{proof}
We use the notations from the proof of Theorem \ref{thm.2.6} replacing $r_{\nu\nu'}:Z_{\nu'}\rightarrow Z_{\nu}$ by $\mathbf{r}_\nu:\mathbf{Z}\rightarrow Z_\nu$ and $u:Y_{\mu}\rightarrow Z_{\nu'}$ by $\mathbf{u}:Y_{\mu}\rightarrow \mathbf{Z}$. Then we have $g_{\nu}\circ q_{\psi(\nu)\mu}=\mathbf{r}_\nu \circ \mathbf{u}$. And by defining $\mathbf{u}'=\mathbf{u}\circ f_{\mu\lambda}:X_{\lambda}\rightarrow \mathbf{Z}$, we obtain $\mathbf{r}_{\nu}\circ \mathbf{u}'=h_{\nu\lambda}$.
\end{proof}
\begin{cor}\label{cor.3.8}Let an inverse system $\mathbf{X}=(X_\lambda,p_{\lambda\lambda'},\Lambda)$ be arbitrary and  $\mathbf{Y}=(Y_\mu,q_{\mu\mu'},M)$
be uniformly movable. Then any morphism $(f_\mu,\phi):\mathbf{X}\rightarrow \mathbf{Y}$ is uniformly movable.
\end{cor}
\begin{proof}
Since $(f_\mu,\phi)=1_{\mathbf{Y}}\circ (f_\mu,\phi)$ and $1_Y:Y\rightarrow Y$ is uniformly movable by Theorem \ref{thm.3.6}, then $(f_\mu,\phi)$ is also uniformly movable according to Theorem \ref{thm.3.7}.
\end{proof}
\begin{thm}\label{thm.3.9}
Let $\mathbf{X}=(X_\lambda,p_{\lambda\lambda'},\Lambda)$, $\mathbf{Y}=(Y_\mu,q_{\mu\mu'},M)$, $\mathbf{Z}=(Z_\nu,r_{\nu\nu'},N)$ be inverse systems in the category $\mathcal{C}$ and $(f_\mu,\phi):\mathbf{X}\rightarrow \mathbf{Y}$, $(g_\nu,\psi):\mathbf{Y}\rightarrow \mathbf{Z}$ morphisms of inverse systems. Suppose that $(f_\mu,\phi)$ is uniformly movable. Then the composition $(h_\nu,\chi)=(g_\nu,\psi)\circ(f_\mu,\phi)$ is also a uniformly movable morphism.
\end{thm}
\begin{proof}
Using the notations from the proof of Theorem \ref{thm.2.8}, there exists $\mathbf{u}:X_\lambda\rightarrow \mathbf{Y}$, such that $f_{\psi(\nu)\lambda}=\mathbf{q}_{\psi(\nu)}\circ \mathbf{u}$. Then for $\mathbf{U}:X_\lambda\rightarrow \mathbf{Z}$,$\mathbf{U}=\mathbf{g}\circ \mathbf{u}$, we have $h_{\nu\lambda}= g_\nu\circ f_{\psi(\nu)\lambda}=\mathbf{r}_\nu\circ \mathbf{U}$.
\end{proof}
\begin{cor}\label{cor.3.10}Let $\mathbf{X}=(X_\lambda,p_{\lambda\lambda'},\Lambda)$ be uniformly movable and $\mathbf{Y}=(Y_\mu,q_{\mu\mu'},M)$ be
arbitrary inverse system. Then any morphism $(f_\mu,\phi):\mathbf{X}\rightarrow \mathbf{Y}$ is uniformly movable.
\end{cor}
\begin{proof}Since $(f_\mu,\phi)=(f_\mu,\phi)\circ 1_{\mathbf{X}}$ and $1_\mathbf{X}:\mathbf{X}\rightarrow \mathbf{X}$ is uniformly movable by Theorem
\ref{thm.3.6}, then $(f_\mu,\phi)$ is also uniformly movable according to Theorem \ref{thm.3.9}.
\end{proof}

\begin{prop}\label{prop.3.11}
Let $(f_\mu,\phi),(f'_\mu,\phi'):\mathbf{X}=(X_\lambda,p_{\lambda\lambda'},\Lambda)\rightarrow \mathbf{Y}=(Y_\mu,q_{\mu\mu'},M)$  be two equivalent morphisms of inverse systems, $(f_\mu,\phi)\sim (f'_\mu,\phi')$. If the morphism $(f_\mu,\phi)$ is uniformly movable then the morphism $(f'_\mu,\phi')$ is also uniformly movable.
\end{prop}
\begin{proof}
We use the notations from the proof of Proposition \ref{prop.2.10}. For $\mu\in M$ consider $\lambda'\in \Lambda$  a uniform movability index with respect to $(f_\mu,\phi)$ such that the relation (2.3) holds. This means $f_{\mu\lambda'}=f'_{\mu\lambda'}$ and therefore $f'_{\mu\lambda'}=f_{\mu\lambda'}=\mathbf{q}_{\mu}\circ \mathbf{u}$.
\end{proof}

Thanks to Proposition \ref{prop.3.11} we can give the following definition.
\begin{defn}\label{defn.3.12}
A morphism in a pro-category pro-$\mathcal{C}$, $\mathbf{f}:\mathbf{X}\rightarrow \mathbf{Y}$,  is called \textsl{uniformly} \textsl{movable} if $\mathbf{f}$ admits a representation $(f_\mu,\phi):\mathbf{X}\rightarrow \mathbf{Y}$ which is uniformly movable.
\end{defn}

The next theorem follows from Theorems \ref{thm.3.7}, \ref{thm.3.9} and Corollaries \ref{cor.3.8}, \ref{cor.3.10}.
\begin{thm}\label{thm.3.13}A composition of an arbitrary pro-morphism at right or left with a uniformly movable pro-morphism is a uniformly movable pro-morphism. Particularly, if for a pro-morphism $\mathbf{f:}\mathbf{X}\rightarrow\mathbf{Y}$, one of $\mathbf{X}$ or $\mathbf{Y}$ is a uniformly movable system, then $\mathbf{f}$ is a uniformly movable pro-morphism.
\end{thm}

\begin{cor}\label{cor.3.14}Let $\mathbf{X}=(X_\lambda,p_{\lambda\lambda'},\Lambda)$ and $\mathbf{Y}=(Y_\mu,q_{\mu\mu'},M)$ in pro-$\mathcal{\mathcal{C}}$. If $\mathbf{Y}$ is uniformly movable and $\mathbf{X}$ is dominated by $\mathbf{Y}$ in pro-$\mathcal{C}$, $\mathbf{X}\leq \mathbf{Y}$, then $\mathbf{X}$ is uniformly movable.
\end{cor}
\begin{proof}
By assumption there exist the pro-morphisms $\mathbf{f}:\mathbf{X}\rightarrow \mathbf{Y}$ and $\mathbf{g}:\mathbf{Y}\rightarrow \mathbf{X}$ such that $\mathbf{g}\circ \mathbf{f}=1_\mathbf{X}$. Because $\mathbf{Y}$ is uniformly movable, by Theorem \ref{thm.3.13} it follows that $\mathbf{f}$, $\mathbf{g}$ and $1_\mathbf{X}$ are uniformly movable pro-morphisms and hence, $\mathbf{X}$ is uniformly movable by Theorem \ref{thm.3.6}.
\end{proof}
\begin{rem}\label{rem.3.15}Corollary \ref{cor.3.14} is a new proof of an important result in the  theory of shape (see \cite{Mard-Segal2}, Theorem 2, Ch. II, \S 6.1).
\end{rem}

\begin{cor}\label{cor.3.16}
Let $\mathbf{f}:\mathbf{X}\rightarrow \mathbf{Y}$ be a pro-morphism with $\mathbf{X}$ an arbitrary inverse system and $\mathbf{Y}=(Y_n,q_{nn+1})$ an inverse sequence. Suppose that $\mathbf{f}$ is movable with a right inverse. Then $\mathbf{f}$ is uniformly movable.
\end{cor}
\begin{proof}
By Theorems \ref{thm.2.12}, \ref{thm.2.5}, $\mathbf{Y}$ is movable. Then by \cite{Mard-Segal2}, Ch. II, \S 6.2, Theorem 4, $\mathbf{Y}$ is uniformly movable and finally, by Theorem \ref{thm.3.13}  the pro-morphism $\mathbf{f}$ is uniformly movable.
\end{proof}
\begin{rem}\label{rem.3.17}
If $F:\mathcal{C}\rightarrow \mathcal{D}$ is a covariant functor and $\mathbf{f}=[(f_\mu,\phi)]:\mathbf{X}=(X_\lambda,p_{\lambda\lambda'},\Lambda)\rightarrow \mathbf{Y}=(Y_\mu,q_{\mu\mu'},M)$ is a morphism in pro-$\mathcal{C}$, then the morphism $F(\mathbf{f})=[(F(f_\mu),\phi)]:F(\mathbf{X})=(F(X_\lambda),F(p_{\lambda\lambda'}),\Lambda)\rightarrow F(\mathbf{Y})=(F(Y_\mu),F(q_{\mu\mu'}),M)$ is well defined and if $\mathbf{f}$ is (strongly, uniformly) movable then $F(\mathbf{f})$ is (strongly, uniformly) movable. For example, if $\mathbf{f}:(\mathbf{X},\mathbf{\ast})\rightarrow (\mathbf{Y},\mathbf{\ast})$ is a (strongly,uniformly) movable pro-morphism in HPol$_*$ then $\pi_1(\mathbf{f}):\pi_1(\mathbf{X},\mathbf{\ast})\rightarrow \pi_1(\mathbf{Y},\mathbf{\ast})$ is a (strongly, uniformly) movable pro-groups morphism.
\end{rem}
Now for to give an important class of uniformly movable morphisms we recall the following definition(\cite{Mard-Segal2},Ch.II, \S 6.1, p.104). Let $\mathcal{C}_0$ be a full subcategory of the category $\mathcal{C}$. An inverse system $\mathbf{X}=(X_\lambda,p_{\lambda\lambda'},\Lambda)\in pro-\mathcal{C}$ is $\mathcal{C}_0$-\textsl{uniformly} \textsl{movable} provided each $\lambda \in \Lambda$ admits a $\lambda'\geq \lambda$ such that for  any object $X_0\in \mathcal{C}_0$, and any morphism $h:X_0\rightarrow X_{\lambda'}$ in $\mathcal{C}$ there exists a morphism $\textbf{r}:X_0\rightarrow \mathbf{X}$ in $\mathcal{C}$ such that
$$ \mathbf{p}_{\lambda}\circ \mathbf{r}=p_{\lambda\lambda'}\circ \mathbf{h}.$$
It follows immediately that the uniform movability in $pro-\mathcal{C}$ implies $\mathcal{C}_0$-uniform movability for every $\mathcal{C}_0$.

The following proposition has the proof analogous to that of Proposition \ref{prop.2.26}.
\begin{prop}\label{pro.3.18} Let $\mathcal{C}_0$ be a full subcategory of the category $\mathcal{C}$ and let $\mathbf{Y}=(Y_\mu,q_{\mu\mu'},M)\in pro- \mathcal{C}$ be a $\mathcal{C}_0$-uniformly movable system. Then if $\mathbf{X}=(X_\lambda,p_{\lambda\lambda'},\Lambda)\in pro-\mathcal{C}_0$, any morphism of inverse systems $(f_\mu,\phi):\mathbf{X}\rightarrow \mathbf{Y}$ is uniformly movable.
\end{prop}

\section{Movable, strongly movable and uniformly movable shape morphisms}

At first we recall some notions on abstract shape theory (cf.\cite{Mard-Segal2}, Ch.I, \S 2.1).

Let $\mathcal{T}$ be a category and $\mathcal{P}$ a subcategory of $\mathcal{T}$. For an object $X\in \mathcal{T}$, a $\mathcal{T}$-\textsl{expansion} of $X$ (with respect to $\mathcal{P}$) is a morphism in pro-$\mathcal{T}$ of $X$  to an inverse system $\mathbf{X}=(X_\lambda,p_{\lambda\lambda'},\Lambda)$ in $\mathcal{T}$, $\mathbf{p}:X\rightarrow \mathbf{X}$ ($X$ viewed as a rudimentary system $(X)$), with the following universal property:

For any inverse system $\mathbf{Y}=(Y_\mu,q_{\mu\mu'},M)$ in the subcategory $\mathcal{P}$ (called a $\mathcal{P}$-\textsl{system})and any morphism $\mathbf{h}:X\rightarrow \mathbf{Y}$ in pro-$\mathcal{T}$, there exists a unique morphism $\mathbf{f}:\mathbf{X}\rightarrow \mathbf{Y}$ in pro-$\mathcal{T}$ such that $\mathbf{h}=\mathbf{f}\circ \mathbf{p}$ , i.e., the following diagram commutes.
$$
 \xymatrix{
& X \ar[d]_{\mathbf{p}}\ar[dr]^{\mathbf{h}} & \\
& {\mathbf{X}} \ar@{-->}[r]_{\mathbf{f}} &{\mathbf{Y}}
 }
$$
Is said that $\mathbf{p}$ is a $\mathcal{P}$-\textsl{expansion} of $X$ provided $\mathbf{X}$ and $\mathbf{f}$ are in pro-$\mathcal{P}$. This case is of particular interest in shape theory.

If $\mathbf{p}:X\rightarrow \mathbf{X}$, $\mathbf{p}':X\rightarrow \mathbf{X}'$ are two $\mathcal{P}$-expansion of the same object $X$, then there is a unique \textsl{natural} \textsl{isomorphism} $\mathbf{i}:\mathbf{X}\rightarrow \mathbf{X}'$ such that $\mathbf{i}\circ \mathbf{p}=\mathbf{p}'$.

Let $\mathcal{T}$ be a category and $\mathcal{P}$ a subcategory. Is said that $\mathcal{P}$ is \textsl{dense} in the category $\mathcal{T}$ provided every object $X\in \mathcal{T}$ admits a $\mathcal{P}$-expansion $\mathbf{p}:X\rightarrow \mathbf{X}$. In \cite{Mard-Segal2} Ch.I, \S 2.2, is proved a theorem which characterizes dense subcategories. For a pair $(\mathcal{T},\mathcal{P})$ consisting by a category $\mathcal{T}$ and a subcategory $\mathcal{P}$ and an object $X\in \mathcal{T}$, the \textsl{comma category of} $X$ \textsl{over} $\mathcal{P}$, denoted by $X_{\mathcal{P}}$ has as objects all morphism $f:X\rightarrow P$ in $\mathcal{T}$, $P\in \mathcal{P}$, and its morphisms $u:f\rightarrow f'$, where $f'\in \mathcal{T}(X,P'), P'\in \mathcal{P}$, are all morphisms $u:P'\rightarrow P$ of $\mathcal{P}$ such that $u\circ f'=f$.
$$
 \xymatrix{
& X \ar[dl]_{f'}\ar[dr]^{f} & \\
{P'} \ar[rr]_{u} &&{P}
 }
$$

\begin{thm}\label{thm.4.1}(\cite{Mard-Segal2}, Ch. I, \S 2.2, Theorem 2)
A subcategory $\mathcal{P}\subseteq \mathcal{T}$ is dense in the category $\mathcal{T}$ if and only if for every object $X\in \mathcal{T}$ the comma category $X_{\mathcal{P}}$ of $X$ over $P$ is cofiltered and cofinally small.
\end{thm}

Let $\mathcal{T}$ be an arbitrary category and $\mathcal{P}\subseteq \mathcal{T}$ a dense subcategory. Let $\mathbf{p}:X\rightarrow \mathbf{X}$ , $\mathbf{p}':X\rightarrow \mathbf{X}'$ be $\mathcal{P}$-expansions of $X\in \mathcal{T}$ and $\mathbf{q}:Y\rightarrow \mathbf{Y}$, $\mathbf{q}':Y\rightarrow \mathbf{Y}'$ $\mathcal{P}$-expansion of $Y\in \mathbf{T}$. Is said that morphisms $\mathbf{f}:\mathbf{X}\rightarrow \mathbf{Y}$ and $\mathbf{f}':\mathbf{X}'\rightarrow \mathbf{Y}'$ in pro-$\mathcal{P}$ are equivalent, $\mathbf{f}\sim \mathbf{f}'$, provided that if $\mathbf{i:X\rightarrow X'}$ and $\mathbf{j:Y\rightarrow Y'}$ are the natural isomorphisms above mentioned, the following diagram commutes:
\[\xymatrix{\mathbf{X}\ar[d]_{\mathbf{f}} \ar[r]^{\mathbf{i}} & {\mathbf{X}'}\ar[d]^{\mathbf{f}'}\\
{\mathbf{Y}}\ar[r]_{\mathbf{j}} & {\mathbf{Y}'}}\]

Now, if $\mathcal{T}$ is a category and $\mathcal{P}$ is a dense subcategory, then a \textsl{shape category} for $(\mathcal{T,P})$, denoted by $\mathbf{Sh}_{(\mathcal{T,P})}$, or simply by $\mathbf{Sh}$, has as objects all objects of $\mathcal{T}$ and morphisms $X\rightarrow Y$ of $\mathbf{Sh}$ are equivalent classes with respect to the relation $\sim $ of morphisms $\mathbf{f:X\rightarrow Y}$ in pro-$\mathcal{P}$, for $\mathbf{p}:X\rightarrow \mathbf{X}$,  $\mathbf{q}:Y\rightarrow \mathbf{Y}$ $\mathcal{P}$ -expansions. Consequently, a \textsl{shape morphism} $F:X\rightarrow Y$ is given by a symbolic diagram:
\[\xymatrix{\textbf{X}\ar[d]_{\textbf{f}}  & X\ar@{-->}[d]^{F}\ar[l]_{\textbf{p}}\\
\textbf{Y} & Y\ar[l]^{\mathbf{{q}}}}\]
Composition of shape morphisms $F:X\rightarrow Y$, $G:Y\rightarrow Z$ is defined by composing representatives $\mathbf{f:X\rightarrow Y}$ with $\mathbf{g:Y\rightarrow Z}$. As indicated above one can always find representatives $\mathbf{f}$ and $\mathbf{g}$ with the target of $\mathbf{f}$ equal to the source of $\mathbf{g}$. The identity shape morphism $1_X:X\rightarrow X$ is defined by $1_{\mathbf{X}}:\mathbf{X\rightarrow X}$. Finally, $\mathbf{Sh}(X,Y)$ is a set because (pro-$\mathcal{P}$)$(\mathbf{X,Y})$ is a set.
\begin{defn}\label{def.4.2}
An object $X$ of a shape category $\mathbf{Sh}_{(\mathcal{T,P})}$ is said to be \textsl{movable} (\textsl{strongly} \textsl{movable},\textsl{uniformly movable}) provided it admits a $\mathcal{P}$-expansion $ \mathbf{p}:X\rightarrow \mathbf{X}=(X_\lambda,p_{\lambda\lambda'},\Lambda)$ such that $\mathbf{X}$ is a movable (strongly movable, uniformly movable) inverse system.
\end{defn}
Based on Theorem \ref{thm.2.12}, with the variant for the strong movability property, and Theorem \ref{thm.3.13}, we can give the following definition.
\begin{defn}\label{def.4.3}
A shape morphism $F:X\rightarrow Y$ is movable (strongly movable, uniformly movable) if it can be represented by a movable (strongly movable, uniformly movable) pro-morphism $\mathbf{f:X\rightarrow Y}$.
\end{defn}

By the same theorems (\ref{thm.2.12}, \ref{thm.3.13}) we have the following comprehensive theorem.
\begin{thm}\label{thm.4.4}
A composition of an arbitrary shape morphism  at right or left with a movable (strongly movable, uniformly movable) shape morphism is a movable
(strongly movable, uniformly movable) shape morphism. Particularly, if for a shape morphism $F:X\rightarrow Y$, one of $X$ or $Y $ is a movable (strongly movable, uniformly movable) shape object, then $F$ is a movable (strongly movable, uniformly movable) shape morphism.
\end{thm}

\begin{cor}\label{cor.4.5}
a) If in the category $\mathbf{Sh}_{(\mathcal{T,P})}$ there exists a movable (strongly movable, uniformly movable) shape morphism $F:X\rightarrow Y$ which admits a right inverse, then $Y$ is a movable (strongly movable, uniformly movable) shape object.

b) If $F:X\rightarrow Y$  is a movable (strictly movable, uniformly movable) shape isomorphism, then $X$ and $Y$  are movable (strongly movable, uniformly movable) shape objects and $F^{-1}$ is a movable (strongly movable, uniformly movable) shape morphism.

c) Consider a commutative diagram in $\mathbf{Sh}_{(\mathcal{T,P})}$
\[\xymatrix{X\ar[d]_{I} \ar[r]^{F} & Y\ar[d]^{J}\\
{X'}\ar[r]_{F'} & {Y'}}\]
with $I$ and $J$ shape isomorphisms and $F'$ a movable (strongly movable, uniformly movable) shape morphism. Then $F$ is movable (strongly movable, uniformly movable) shape morphism.
\end{cor}

Let $\mathbf{Sh}_{(\mathcal{T,P})}$ be a shape category, $f\in \mathcal{T}(X,Y)$ and $\mathbf{p}:X\rightarrow \mathbf{X}$, $\mathbf{q}:Y\rightarrow \mathbf{Y}$ $P$-expansions. Then $f$ induces a unique morphism $\mathbf{f}:\mathbf{X\rightarrow Y}$ in pro-$\mathcal{P}$ such that the following diagram commutes
\[\xymatrix{\textbf{X}\ar[d]_{\textbf{f}}  & X\ar[d]^{f}\ar[l]_{\textbf{p}}\\
\textbf{Y} & Y\ar[l]^{\mathbf{{q}}}}\]
If we take other $\mathcal{P}$-expansions $\mathbf{p}':X\rightarrow \mathbf{X}'$, $\mathbf{q}':Y\rightarrow \mathbf{Y}'$, there is another morphism $\mathbf{f}':\mathbf{X}'\rightarrow \mathbf{Y}'$ in pro-$\mathcal{P}$ such that $ \mathbf{f}'\circ \mathbf{p}'=\mathbf{q}'\circ f$. Then, we can write $\mathbf{i}\circ \mathbf{p}=\mathbf{p}'$, $\mathbf{j}\circ \mathbf{q}=\mathbf{q}'$ and this implies $(\mathbf{f}'\circ \mathbf{i})\circ \mathbf{p}=\mathbf{f}'\circ \mathbf{p}'=\mathbf{q}'\circ f=\mathbf{j}\circ \mathbf{q}\circ f=(\mathbf{j}\circ f)\circ \mathbf{p}$, which by uniqueness, yields $\mathbf{f}'\circ \mathbf{i}=\mathbf{j}\circ \mathbf{f}$, with $\mathbf{i},\mathbf{j}$ isomorphisms. Hence, $\mathbf{f}\sim \mathbf{f}'$. In this way $f\in \mathcal{T}(X,Y)$ defines a shape morphism, $S(f):=[\mathbf{f}]\in \mathbf{Sh}(X,Y)$, and by putting $S(X)=X$ it is obtained a covariant functor $S:\mathcal{T}\rightarrow \mathbf{Sh}_{(\mathcal{T,P})}$, called the \textsl{shape functor}.

\begin{defn}\label{def.4.6}

Let $\mathbf{Sh}_{(\mathcal{T,P})}$ be a shape category and let $f\in \mathcal{T}(X,Y)$ be a morphism in $\mathcal{T}$. Then we say that $f$ is \textsl{shape movable (strongly movable, uniformly movable) morphism }if $F:=S(f)\in \mathbf{Sh}_{(\mathcal{T,P})}(X,Y)$ is a movable (strongly movable, uniformly movable) shape morphism.
\end{defn}

\begin{cor}\label{cor.4.7}
Let $\mathbf{Sh}_{(\mathcal{T,P})}$ be a shape category.

a) A composition in the category $\mathcal{T}$ of an arbitrary morphism  at right or left with a movable (strongly movable, uniformly movable ) morphism is a movable(strongly movable, uniformly movable) morphism. Particularly, if for a morphism $f:X\rightarrow Y$, one of $X$ or $Y $ is a movable(strongly movable, uniformly movable) shape object, for example $X\in \mathcal{P}$ or $Y\in \mathcal{P}$, then $f$ is a movable (strongly movable, uniformly movable) morphism.

b) If in the category $\mathcal{T}$ there exists a movable (strongly movable, uniformly movable) morphism $f:X\rightarrow Y$ which admits a right inverse, then $Y$ is a movable (strongly movable, uniformly movable) shape object.

c)If $f:X\rightarrow Y$  is a movable(strongly movable, uniformly movable) isomorphism in $\mathcal{T}$, then $X$ and $Y$  are movable (strongly movable, uniformly movable) shape objects and $f^{-1}$ is a movable (strongly movable, uniformly movable) morphism.
\end{cor}

\begin{rem}\label{rem.4.8}
Suppose that $F\in \mathbf{Sh}_{(\mathcal{T},\mathcal{P})}(X,Y)$ is given of a morphism $\mathbf{f}=[(f_\mu,\phi)]:\mathbf{X}=(X_\lambda,p_{\lambda\lambda'},\Lambda)\rightarrow \mathbf{Y}=(Y_\mu,q_{\mu\mu'},M)$, with $\mathbf{p}=(p_\lambda):X \rightarrow \mathbf{X}$ a $\mathcal{P}$-expansion of $X$. Then after an idea in \cite{Ava-Gev} (Proposition 4) we can prove that the condition (2.10) from the definition of the strong movability can be written as
\begin{equation}
u\circ p_{\lambda}=f_{\mu'}\circ p_{\phi(\mu')}.
\end{equation}
Indeed, if (2.10) holds, i.e., $u\circ p_{\lambda\lambda^\ast}=f_{\mu'}\circ p_{\phi(\mu')\lambda^\ast}$, then  composing at right by $p_{\lambda^\ast}$, we obtain (4.1). Conversely, suppose that the relation (4.1) is satisfied. Consider $\lambda^0\in \Lambda, \lambda^0\geq \lambda,\phi(\mu')$, for which we have $(u\circ p_{\lambda\lambda^0})\circ p_{\lambda^0} =(f_{\mu'}\circ p_{\phi(\mu')\lambda^0})\circ p_{\lambda^0}$. Then by \cite{Mard-Segal1}, Ch.I, \S 2.1, Theorem 1, condition (AE2), there is an index $\lambda^\ast\geq \lambda^0$ such that $u\circ p_{\lambda\lambda^0}\circ p_{\lambda^0\lambda^\ast}=f_{\mu'}\circ p_{\phi(\mu')\lambda^0}\circ p_{\lambda^0\lambda^\ast}$, which implies (2.10).
\end{rem}

\section{Co-movability}

\begin{defn}\label{def.5.1}

Let $\mathbf{X}=(X_\lambda,p_{\lambda\lambda'},\Lambda)$ and $\mathbf{Y}=(Y_\mu,q_{\mu\mu'},M)$ be inverse systems in a category $\mathcal{C}$ and $(f_\mu,\phi):\mathbf{X}\rightarrow \mathbf{Y}$ a morphism of inverse systems. We say that the $(f_\mu,\phi)$ is \textsl{co-movable morphism} provided every $\mu\in M$ admits $\lambda\in \Lambda$, $\lambda\geq \phi(\mu)$ (called a \textsl{co-movability index} of $\mu$ relative to $(f_\mu,\phi)$ ) such that each $\lambda'\geq \phi(\mu)$ admits a morphism $r:X_{\lambda}\rightarrow X_{\lambda'}$ of $\mathcal{C}$ which satisfies \begin{equation}
f_{\mu\lambda}=f_{\mu\lambda'}\circ r,
\end{equation}
i.e., makes the following outside diagram commutative
$$\xymatrix{
& {Y_\mu} &\\
& {X_{\phi(\mu)}}\ar[u]^-{f_\mu} &\\
{X_{\lambda}}\ar[uur]^-{f_{\mu\lambda}}\ar[ur]_-{p_{\phi(\mu)\lambda}}\ar@{-->}[rr]_-{r} & & {X_{\lambda'}}\ar[ul]^-{p_{\phi(\mu)\lambda'}}\ar[uul]_-{f_{\mu\lambda'}}
}$$

If in addition there exists a index $\lambda^\ast\geq \lambda,\lambda'$ such that
\begin{equation}
r\circ p_{\lambda\lambda^\ast}=p_{\lambda'\lambda^\ast},
\end{equation}
we say that $(f_\mu,\phi)$ is \textsl{strongly co-movable morphism}. In this case $\lambda$ is called a \textsl{strong }\textsl{ co-movability index} for $\mu$ with respect to $(f_\mu,\phi)$.
\end{defn}

\begin{defn}\label{def.5.2}

Let $\mathbf{X}=(X_\lambda,p_{\lambda\lambda'},\Lambda)$ and $\mathbf{Y}=(Y_\mu,q_{\mu\mu'},M)$ be inverse systems in a  category $\mathcal{C}$ and $(f_\mu,\phi):\mathbf{X}\rightarrow \mathbf{Y}$ a morphism of inverse systems. We say that the $(f_\mu,\phi)$ is \textsl{uniformly }\textsl{co-movable morphism} provided every $\mu\in M$ admits $\lambda\in \Lambda$, $\lambda\geq \phi(\mu)$ (called a \textsl{uniform} \textsl{co-movability index} of $\mu$ relative to $(f_\mu,\phi)$ ) such that there is a morphism $\mathbf{r}:X_{\lambda}\rightarrow \mathbf{X}$ in pro- $\mathcal{C}$ satisfying
\begin{equation}
f_{\mu\lambda}=\mathbf{f}_{\mu}\circ \mathbf{r},
\end{equation}
i.e., makes the following outside diagram commutative
$$\xymatrix{
& {Y_\mu} &\\
& {X_{\phi(\mu)}}\ar[u]^-{f_\mu} &\\
{X_{\lambda}}\ar[uur]^-{f_{\mu\lambda}}\ar[ur]_-{p_{\phi(\mu)\lambda}}\ar@{-->}[rr]_-{\mathbf{r}} & & {\mathbf{X}}\ar[ul]^-{\mathbf{p}_{\phi(\mu)}}\ar[uul]_-{\mathbf{f}_{\mu}}
}$$

where $\mathbf{f}_\mu=f_\mu \circ \mathbf{p}_{\phi(\mu)}$.
\end{defn}
\begin{rem}\label{rem.5.3}
The pro-morphism $\mathbf{r}:X_{\lambda}\rightarrow \mathbf{X}$ is given by some morphisms $r_{\lambda'}:X_{\lambda}\rightarrow X_{\lambda'}$, such that if $\lambda'_1\leq \lambda'_2$ then $r_{\lambda'_1}=p_{\lambda'_1\lambda'_2}\circ r_{\lambda'_2}$. The relation $f_{\mu\lambda}=\mathbf{f}_\mu\circ \mathbf{r}$ means $f_{\mu\lambda}=f_\mu\circ r_{\phi(\mu)}$. Therefore, $\lambda'\geq \phi(\mu)$ implies $f_{\mu\lambda}=f_\mu\circ p_{\phi(\mu)\lambda'}\circ r_{\lambda'}=f_{\mu\lambda'}\circ r_{\lambda'}$. In this way we have that uniform co-movability implies co-movability.
\end{rem}
\begin{rem}\label{rem.5.4} Suppose that $(f_\mu,\phi):\mathbf{X}=(X_\lambda,p_{\lambda\lambda'},\Lambda)\rightarrow \mathbf{Y}=(Y_\mu,q_{\mu\mu'},M)$ is a morphism of inverse systems and let $\mu\in M$. If $\lambda \geq \phi(\mu)$ is a co-movability (strong co-movability, uniform co-movability) index of $\mu$ with respect to $(f_\mu,\phi)$, then any $\lambda'>\lambda$ has this property.

\end{rem}
\begin{prop}\label{prop.5.5}
Let $\mathbf{X}=(X_\lambda,p_{\lambda\lambda'},\Lambda)$ and $\mathbf{Y}=(Y_\mu,q_{\mu\mu'},M)$ inverse systems in a the category $\mathcal{C}$ and $(f_\mu,\phi):\mathbf{X}\rightarrow \mathbf{Y}$ a morphism of inverse systems. If $\mathbf{X}$ is a movable (strongly movable, uniformly movable) system and $\mathbf{Y}$ is an arbitrary system, then $(f_\mu,\phi)$ is co-movable (strongly co-movable, uniformly co-movable) morphism.
\end{prop}
\begin{proof}
If $\mu\in M$,  then a movability (strong movability,uniform movability) index $\lambda'\in \Lambda$ for $\phi(\mu)$ this  is a co-movability (strong co-movability, uniform co-movability)index for $\mu$ with respect to $(f_\mu,\phi)$.
\end{proof}
\begin{prop}\label{prop.5.6}Let  $(f_\mu,\phi):\mathbf{X}=(X_\lambda,p_{\lambda\lambda'},\Lambda)\rightarrow \mathbf{Y}=(Y_\mu,q_{\mu\mu'},M)$ be  a morphism of inverse systems such that the directed preordered set $(M,\leq)$ admits a cofinal subset $\overline{M}$ such that every $\overline{\mu}\in \overline{M}$ admits a co-movability (strong co-movability, uniform co-movability) index with respect to $(f_\mu,\phi)$. Suppose also that $\phi:(M,\leq)\rightarrow (\Lambda,\leq)$ is a decreasing function. Then $(f_\mu,\phi)$ is a co-movable (strongly co-movable, uniformly co-movable) morphism.
\end{prop}
\begin{proof}
In the given assumptions each $\mu \in M$ admits a $\overline{\mu}\in \overline{M}$ such that $\mu\leq\overline{\mu}$. Consider the case of simple co-movability. We must show that there exists a co-movability index for $\mu$. By Remark \ref{rem.5.4}, we can choose for $\overline{\mu}\in \overline{M}, \overline{\mu}\geq \mu$, a co-movability index $\lambda\geq \phi(\mu)\geq \phi(\overline{\mu})$ such that $f_{\mu\lambda}=q_{\mu\overline{\mu}}\circ
f_{\overline{\mu} \lambda}$. If $\lambda'\geq\phi(\mu)\geq \phi(\overline{\mu})$, consider $\overline{\lambda'}\geq \lambda'$ and such that $f_{\mu\overline{\lambda'}}=q_{\mu\overline{\mu}}\circ f_{\overline{\mu}{\overline{\lambda'}}}$. For this let $\overline{r }:X_\lambda\rightarrow X_{\overline{\lambda'}}$ be such that $f_{\overline{\mu}\lambda}=f_{\overline{\mu}\overline{\lambda'}}\circ \overline{r}$. Then $f_{\mu\lambda}=q_{\mu\overline{\mu}}\circ f_{\overline{\mu}\overline{\lambda'}}\circ \overline{r}=f_{\mu\overline{\lambda'}}\circ \overline{r} =f_{\mu\lambda'}\circ(p_{\lambda'\overline{\lambda'}}\circ \overline{r})=f_{\mu\lambda'}\circ r$. If $\overline{r}\circ p_{\lambda\lambda^\ast}=
p_{\lambda'\lambda^\ast}$, then $r\circ p_{\lambda\lambda^\ast}=p_{\lambda'\lambda^\ast}$. The proof in the case of the uniform co-movability is easy adapted from the above.
\end{proof}
\begin{thm}\label{thm.5.7}An inverse system $\mathbf{X}=(X_\lambda,p_{\lambda\lambda'},\Lambda)$ is movable (strongly movable, uniformly movable) if and only if the identity morphism $1_\mathbf{X}$ is co-movable (strongly co-movable, uniformly co-movable).
\end{thm}
\begin{proof}If $\mathbf{X}$ is movable (strongly movable, uniformly movable), then by Prop. \ref{prop.5.5}, the morphism $1_{\mathbf{X}}$ is co-movable (strongly co-movable, uniformly co-movable). Conversely, if $1_\mathbf{X}$ is co-movable (strongly co-movable, uniformly co-movable), then a co-movability (strong co-movability, uniform co-movability) index of $\lambda\in \Lambda$ with respect to $1_{\mathbf{X}}=(1_{X_\lambda},id_{\Lambda})$ is a movability (strong movability, uniform movability)index for $\lambda$.
\end{proof}
\begin{thm}\label{thm.5.8}A morphism of inverse systems of pointed sets is co-movable if and only if it has the Mittag-Leffler property.
\end{thm}
\begin{proof}Let $(f_\mu,\phi):((X_\lambda,\ast),p_{\lambda\lambda'},\Lambda)\rightarrow ((Y_\mu,\ast),q_{\mu\mu'},M)$ be a morphism with the Mittag-Leffler property. Then for $\mu\in M$ there is a ML index $\lambda\in \Lambda$, $\lambda\geq\phi(\mu)$ such that
(2.4) holds for each $\lambda'\geq \lambda$. We can prove that $\lambda$ is a co-movability index of $\mu$ with respect to $(f_\mu,\phi)$. If $\lambda'
\geq \phi(\mu)$ and $\lambda'\geq \lambda$, the relation (2.4) defines a map of pointed sets $r:(X_\lambda,\ast)\rightarrow (X_{\lambda'},\ast)$ such
that $f_{\mu\lambda'}\circ r=f_{\mu\lambda}$. For any other $\lambda''\geq \phi(\mu)$, one choose $\lambda'''\geq \lambda'',\phi(\mu)$ and consider
$r':X_{\lambda}\rightarrow X_{\lambda'''}$ such that $f_{\mu\lambda'''}\circ r'=f_{\mu\lambda}$. Then the composition $r:=p_{\lambda''\lambda'''}\circ r'$ satisfies the relation $f_{\mu\lambda''}\circ r=f_{\mu\lambda''}\circ p_{\lambda''\lambda'''}\circ r'=f_{\mu\lambda'''}\circ r'=f_{\mu\lambda}$.

Conversely, let $(f_\mu,\phi)$ be co-movable. Let $\mu\in M$ and $\lambda\in \Lambda$ with $\lambda\geq \phi(\mu)$ a co-movability index of $\mu$ with respect to $(f_\mu,\phi)$. Then, for $\lambda'\geq \lambda$ there exists $r:(X_\lambda,\ast)\rightarrow (X_{\lambda'},\ast)$ such that $f_{\mu\lambda'}
\circ r=f_{\mu\lambda}$. This implies the inclusion $f_{\mu\lambda}(X_\lambda)\subseteq f_{\mu\lambda'}(X_{\lambda'})$. The converse inclusion follows from the relation $f_{\mu\lambda}\circ p_{\lambda\lambda'}=f_{\mu\lambda'}$.
\end{proof}
Using Theorems \ref{thm.5.7}, \ref{thm.5.8} and Definition \ref{def.2.15},  we obtain the following corollary (see \cite{Mard-Segal2}, Ch. II, \S 6.1, Theorem 2).
\begin{cor}\label{cor.5.9}A pointed pro-set $(\mathbf{X},\ast)\in$ pro-Set$_\ast$ is movable if and only if it has the Mittag-Leffler property.
\end{cor}

\begin{prop}\label{prop.5.10}
Let $(f_\mu,\phi),(f'_\mu,\phi'):\mathbf{X}=(X_\lambda,p_{\lambda\lambda'},\Lambda)\rightarrow \mathbf{Y}=(Y_\mu,q_{\mu\mu'},M)$ be two equivalent morphisms of inverse systems, $(f_\mu,\phi)\sim (f'_\mu,\phi')$. If the morphism $(f_\mu,\phi)$ is co-movable (strongly co-movable) then the morphism $(f'_\mu,\phi')$ is also co-movable (strongly co-movable).
\end{prop}
\begin{proof}
Let $\mu\in M$ with $\lambda\in \Lambda, \lambda\geq \phi(\mu)$ a co-movability index for $\mu$ relative to $(f_\mu,\phi)$. Consider $\overline{\lambda}\in \Lambda$ sufficiently large, $\overline{\lambda}\geq \lambda$ and $\overline{\lambda}\geq \phi(\mu),\phi'(\mu)$ such that

$$f_{\mu\overline{\lambda}}=f'_{\mu\overline{\lambda}}$$.

We can prove that $\overline{\lambda}$ is a co-movability index of $\mu$ relative to $(f'_\mu,\phi')$. If $\overline{\lambda'}\geq \phi'(\mu)$, consider $\lambda'\geq \overline{\lambda'},\phi(\mu)$. For this there exists a morphism $r:X_{\lambda}\rightarrow X_{\lambda'}$ such that
$$ f_{\mu\lambda}=f_{\mu\lambda'}\circ r.$$
Consider the morphism $\overline{r}:X_{\overline{\lambda}}\rightarrow X_{\overline{\lambda'}}$ defined as
$$ \overline{r}=p_{\overline{\lambda'}\lambda'}\circ r \circ p_{\lambda\overline{\lambda}} $$.

Then $f'_{\mu\overline{\lambda'}}\circ \overline{r}=f'_\mu\circ p_{\phi'(\mu)\overline{\lambda'}}\circ \overline{r}= f'_\mu\circ p_{\phi'(\mu)\overline{\lambda'}}\circ p_{\overline{\lambda'}\lambda'}\circ r\circ p_{\lambda\overline{\lambda}}=f'_\mu\circ p_{\phi'(\mu)\lambda'}\circ r\circ p_{\lambda\overline{\lambda}}=
f_\mu\circ p_{phi(\mu)\lambda'}\circ r\circ p_{\lambda\overline{\lambda}}=f_\mu\circ p_{\phi(\mu)\lambda}\circ p_{\lambda\overline{\lambda}}=f_\mu\circ p_{\phi(\mu)\overline{\lambda}}=f'_\mu\circ p_{\phi'(\mu)\overline{\lambda}}$, i.e., $f'_{\mu\overline{\lambda'}}\circ \overline{r}=f'_{\mu\overline{\lambda}}$.

If $\lambda$ is a strong co-movability index for $\mu$ with respect to $(f_\mu,\phi)$ such that $r\circ p_{\lambda\lambda^\ast}=p_{\lambda'\lambda^\ast}$, for $\lambda^\ast\geq \lambda,\lambda'$, then consider $\overline{\lambda^\ast}\geq \overline{\lambda},\overline{\lambda'},\lambda^\ast$. For this index we have: $\overline{r}\circ p_{\overline{\lambda}\overline{\lambda^\ast}}=p_{\overline{\lambda'}\lambda'}\circ r\circ p_{\lambda\overline{\lambda}}\circ p_{\overline{\lambda}\overline{\lambda^\ast}}=p_{\overline{\lambda'}\lambda'}\circ r\circ p_{\lambda\overline{\lambda^\ast}}=p_{\overline{\lambda'}\lambda'}\circ r\circ p_{\lambda\lambda^\ast}\circ p_{\lambda^\ast\overline{\lambda^\ast}}=p_{\overline{\lambda'}\lambda'}\circ p_{\lambda'\lambda^\ast}\circ p_{\lambda^\ast\overline{\lambda^\ast}}=p_{\overline{\lambda'} \overline{\lambda^\ast}}$.
\end{proof}
\begin{prop}\label{prop.5.11}
Let $(f_\mu,\phi),(f'_\mu,\phi'):\mathbf{X}=(X_\lambda,p_{\lambda\lambda'},\Lambda)\rightarrow \mathbf{Y}=(Y_\mu,q_{\mu\mu'},M)$ be two equivalent morphisms of inverse systems, $(f_\mu,\phi)\sim (f'_\mu,\phi')$. If the morphism $(f_\mu,\phi)$ is uniformly co-movable then the morphism $(f'_\mu,\phi')$ is also uniformly co-movable.
\end{prop}
\begin{proof}
We use the notations from the proof of Proposition \ref{prop.5.10} with the necessary changes: $\lambda$ is a uniform co-movability index for $\mu$ relative to $(f_\mu,\phi)$ and we prove that $\overline{\lambda}$ is a uniform co-movability index of $\mu$ relative to $(f'_\mu,\phi')$. The morphism $r$ is replaced by the pro-morphism $\mathbf{r}:X_{\lambda}\rightarrow \mathbf{X}$ satisfying $f_{\mu\lambda}=\mathbf{f}_{\mu}\circ \mathbf{r}$. Then defining $\overline{\mathbf{r}}:X_{\overline{\lambda}}\rightarrow \mathbf{X}$ by $\overline{\mathbf{r}}=\mathbf{r}\circ p_{\lambda\overline{\lambda}}$  it is verified the relation $f'_{\mu\overline{\lambda}}=\mathbf{f}'_\mu \circ \overline{\mathbf{r}}$.
\end{proof}
Thanks to Propositions \ref{prop.5.10} and \ref{prop.5.11}  we can give the following definition.
\begin{defn}\label{defn.5.12}
A morphism in a pro-category pro-$\mathcal{C}$, $\mathbf{f}:\mathbf{X}\rightarrow \mathbf{Y}$,  is called \textsl{co-movable (strongly co-movable, uniformly co-movable)} if $\mathbf{f}$ admits a representation $(f_\mu,\phi):\mathbf{X}\rightarrow \mathbf{Y}$ which is co-movable (strongly co-movable, uniformly co-movable) morphism.
\end{defn}
\begin{prop}\label{prop.5.13}

Let $\mathbf{X}=(X_\lambda,p_{\lambda\lambda'},\Lambda)$, $\mathbf{Y}=(Y_\mu,q_{\mu\mu'},M)$, $\mathbf{Z}=(Z_\nu,r_{\nu\nu'},N)$ be inverse systems in the category $\mathcal{C}$ and $\mathbf{f}=[(f_\mu,\phi)]:\mathbf{X}\rightarrow \mathbf{Y}$, $\mathbf{g}=[(g_\nu,\psi)]:\mathbf{Y}\rightarrow \mathbf{Z}$ pro-morphisms. Suppose that $\mathbf{f}$ is a co-movable (strongly co-movable, uniformly co-movable) pro-morphism. Then the composition $\mathbf{h}=\mathbf{g}\circ \mathbf{f}$,  $\mathbf{h}=[(h_\nu,\chi)]=[(g_\nu\circ f_{\psi(\nu)},\phi\circ \psi)]$ is also a co-movable (strongly co-movable, uniformly co-movable) pro-morphism.
\end{prop}
\begin{proof}
Suppose that $(f_\mu,\phi)$ is co-movable. Let $\nu\in N$ an arbitrary index. Consider $\psi(\nu)\in M$ and denote by $\lambda\in \Lambda$, $\lambda\geq \phi(\psi(\nu))$ a co-movability index of $\psi(\nu)$ relative to $(f_\mu,\phi)$. We can prove that $\lambda$ is a co-movability index of $\nu$ relative to $(h_\nu,\chi)$. Indeed, if $\lambda'\geq \chi(\nu)=\phi(\psi(\nu))$ then there is a morphism $r:X_{\lambda}\rightarrow X_{\lambda'}$ such that
$$ f_{\psi(\nu)\lambda}=f_{\psi(\nu)\lambda'}\circ r.$$
Then $h_{\nu\lambda'}\circ r=h_\nu\circ p_{\phi(\psi(\nu))\lambda'}\circ r=g_\nu\circ (f_{\psi(\nu)}\circ p_{\phi(\psi(\nu))\lambda'})\circ r= g_\nu\circ (f_{\psi(\nu)\lambda'}\circ r)=g_\nu\circ f_{\psi(\nu)\lambda}=(g_\nu\circ f_{\psi(\nu)})\circ p_{\phi(\psi(\nu))\lambda}=h_\nu\circ p_{\chi(\nu)\lambda}=h_{\nu\lambda}$.

Obviously that if $\lambda$ is strong co-movability index of $\psi(\nu)$ relative to $(f_\mu,\phi)$ it is also a strong co-movability index of $\nu$ relative to $(h_\nu,\chi)$.

The case of uniform co-movability is absolutely similar.
\end{proof}
\begin{prop}\label{prop.5.14}
Let in the category pro-$\mathcal{C}$ be the morphisms $\mathbf{f}=[(f_\mu,\phi)]:\mathbf{X}=(X_\lambda,p_{\lambda\lambda'},\Lambda)\rightarrow \mathbf{Y}=(Y_\mu,q_{\mu,\mu'},M)$, $\mathbf{g}=[(g_\nu,\psi)]:\mathbf{Y}\rightarrow \mathbf{Z}=(Z_\nu,r_{\nu\nu'},N)$  and $\mathbf{h}=\mathbf{g}\circ \mathbf{f}=[(h_\nu=g_\nu\circ f_{\phi(\mu)},\chi=\phi\circ\psi)]:\mathbf{X}\rightarrow \mathbf{Z}$. Suppose that $\mathbf{h}$ is a co-movable (strictly co-movable, uniformly co-movable) pro-morphism and that $\mathbf{f}$ admits a right inverse $\mathbf{f}'=[(f'_\lambda,\phi')]:\mathbf{Y}\rightarrow \mathbf{X}$, $\mathbf{f}\circ \mathbf{f'}=1_{\mathbf{Y}}$. Then $\mathbf{g}$ is a co-movable (strictly co-movable, uniformly co-movable) pro-morphism.
\end{prop}
\begin{proof}
We give in detail the proof for the co-movability case. Consider $\nu\in N$ and let $\lambda\in \Lambda$, $\lambda\geq \phi(\psi(\nu))$ a co-movability index for $\nu$ relative to $(h_\nu,\chi)$. If we consider an index $\mu\in M, \mu\geq \psi(\nu),\phi'(\lambda')$, then we can prove that $\mu$ is a co-movability index for $\nu$ relative to $(g_\nu,\psi)$. If $\mu'\in M, \mu'\geq \psi(\nu)$, consider $\lambda'\in \Lambda$, $\lambda'\geq \phi(\mu'),\phi(\psi(\nu))$. Then there exists a morphism $r:X_{\lambda}\rightarrow X_{\lambda'}$ such that $h_{\nu\lambda}=h_{\nu\lambda'}\circ r$, i.e.,
$$ g_\nu\circ f_{\psi(\nu)}\circ p_{\phi(\psi(\nu))\lambda}=g_\nu\circ f_{\psi(\nu)}\circ p_{\phi(\psi(\nu))\lambda'} \circ r.  $$
Now we define $\rho:Y_{\mu}\rightarrow Y_{\mu'}$ by
$$ \rho:= f_{\mu'}\circ p_{\phi(\mu')\lambda'}\circ r\circ f'_{\lambda}\circ q_{\phi'(\lambda)\mu} .  $$
For this we have: $g_{\nu\mu'}\circ \rho=g_\nu\circ q_{\psi(\nu)\mu'}\circ f_{\mu'}\circ p_{\phi(\mu')\lambda'}\circ r\circ f'_{\lambda}\circ q_{\phi'(\lambda)\mu}$. Now if we suppose $\lambda'$ sufficiently large and we transcribe the relation $\mathbf{f}\circ \mathbf{f}'=1_{\mathbf{Y}}$ by
$$f_{\psi(\nu)}\circ p_{\phi(\psi(\nu))\lambda}\circ f'_{\lambda}\circ q_{\phi'(\lambda)\mu}=q_{\psi(\nu)\mu}, $$  we continue as $g_{\nu\mu'}\circ \rho=g_\nu\circ f_{\psi(\nu)}\circ p_{\phi(\psi(\nu))\lambda}\circ f'_{\lambda}\circ q_{\phi'(\lambda)\mu}=g_\nu\circ q_{\psi(\nu)\mu}=g_{\nu\mu}$.

If there exists $\lambda^\ast\geq \lambda,\lambda'$ such that $r\circ p_{\lambda\lambda^\ast}=p_{\lambda'\lambda^\ast}$ , consider $\mu^\ast\geq \mu,\mu',\phi'(\lambda^\ast)$. For this index we have:$\rho\circ q_{\mu\mu^\ast}=f_{\mu'}\circ p_{\phi(\mu')\lambda'}\circ r\circ f'_{\lambda}\circ q_{\phi(\lambda)\mu}\circ q_{\mu\mu^\ast}=f_{\mu'}\circ p_{\phi(\mu')\lambda'}\circ r\circ f'_{\lambda}\circ q_{\phi'(\lambda)\mu^\ast}=f_{\mu'}\circ p_{\phi(\mu')\lambda'}\circ p_{\lambda\lambda^\ast}\circ p_{\lambda\lambda^\ast}\circ f'_{\lambda^\ast}\circ q_{\phi'(\lambda^\ast)\mu^\ast}=f_{\mu'}\circ p_{\phi(\mu')\lambda^\ast}\circ f'_{\lambda^\ast}\circ q_{\phi'(\lambda^\ast)\mu^\ast}=q_{\mu'\mu^\ast}$.

In the case of uniform co-movability, by hypothesis we have $h_{\nu\lambda}=\mathbf{h}_\nu\circ \mathbf{r}$ for $\mathbf{r}:X_{\lambda}\rightarrow \mathbf{X}$ a pro-morphism, and we define $\mathbf{\rho}:Y_{\mu}\rightarrow \mathbf{Y}$ by $\mathbf{\rho}=\mathbf{f}\circ \mathbf{r}\circ f'_{\lambda}\circ q_{\phi'(\lambda)\mu}$. For this we obtain $g_{\nu\mu}=\mathbf{g}_\nu \circ \mathbf{\rho}$.
\end{proof}
\begin{cor}\label{cor.5.15}
Let in the category pro-$\mathcal{C}$ be the composition  $\mathbf{h}=\mathbf{g}\circ \mathbf{f}$ . If $\mathbf{g}$ is a co-movable(strongly co-movable, uniformly co-movable)pro-morphism and $\mathbf{f}$ is an isomorphism, then $\mathbf{h}$ is a co-movable(strongly co-movable, uniformly co-movable)  pro-morphism.
\end{cor}
\begin{proof}
We write $\mathbf{g}=\mathbf{h}\circ \mathbf{f}^{-1} $ and apply Proposition \ref{prop.5.13}.
\end{proof}
\begin{cor}\label{cor.5.16}
If there exists a  co-movable (strongly co-movable, uniformly co-movable)pro-morphism  $\mathbf{f}:\mathbf{X}\rightarrow \mathbf{Y}$ which admits a left inverse, then $\mathbf{X}$ is movable (strongly movable, uniformly movable) system. If $\mathbf{f}$ is an isomorphism, then $\mathbf{Y}$ is also movable (strongly movable, uniformly movable) system and $\mathbf{f}^{-1}$ is a co-movable (strongly co-movable, uniformly co-movable) pro-morphism.
\end{cor}
\begin{cor}\label{cor.5.17}
Consider a commutative diagram in pro-$\mathcal{C}$
\[\xymatrix{\mathbf{X}\ar[d]_{\mathbf{i}} \ar[r]^{\mathbf{f}} & {\mathbf{Y}}\ar[d]^{\mathbf{j}}\\
{\mathbf{X}'}\ar[r]_{\mathbf{f}'} & {\mathbf{Y}'}}\]
with $\mathbf{i}$ and $\mathbf{j}$ isomorphisms and $\mathbf{f}$ a co-movable (strongly co-movable, uniformly co-movable) morphism. Then $\mathbf{f'}$ is co-movable (strongly co-movable, uniformly co-movable)pro-morphism. .
\end{cor}
By Corollary \ref{cor.5.17} we can formulate the following definition.
\begin{defn}\label{def.5.18}
A shape morphism $F:X\rightarrow Y$ is \textsl{co-movable} (\textsl{strongly co-movable}, \textsl{uniformly co-movable}) if it can be represented by a co-movable (strongly co-movable, uniformly co-movable) pro-morphism $\mathbf{f:X\rightarrow Y}$.
\end{defn}
\begin{defn}\label{def.5.19}
Let $\mathbf{Sh}_{(\mathcal{T,P})}$ be a shape category and let $f\in \mathcal{T(}X,Y)$ be a morphism in $\mathcal{T}$. Then we say that $f$ is \textsl{shape co-movable (strongly co-movable, uniformly co-movable) morphism }if $F:=S(f)\in \mathbf{Sh}_{(\mathcal{T,P})}(X,Y)$ is a co-movable (strongly co-movable, uniformly co-movable) shape morphism.
\end{defn}
\begin{rem}\label{rem.5.21}
In the conditions of Remark \ref{rem.4.8}, the condition (5.2) of strong co-movability can be written as
\begin{equation}
r\circ p_{\lambda}=p_{\lambda'}
\end{equation}
\end{rem}
\begin{rem}\label{rem.5.22}
All properties of co-movability, strong co-movability  and uniform co-movability of morphism of inverse systems and of pro-morphisms can be  transferred to appropriate properties for shape morphisms and for morphisms in the category $\mathcal{T}$ of a shape theory $\mathbf{Sh}_{(\mathcal{T,P})}$.
\end{rem}

\section{Applications}

If $\mathbf{p}:X\rightarrow \mathbf{X},\mathbf{q}:Y\rightarrow \mathbf{Y}$ are inverse limits of $\mathbf{X},\mathbf{Y}\in $pro-$\mathcal{C}$, then a morphism $\mathbf{f}:\mathbf{X}\rightarrow \mathbf{Y}$ in pro-$\mathcal{C}$ induces a morphism $f:X\rightarrow Y$ in $\mathcal{C}$ satisfying $\mathbf{q}\circ f=\mathbf{f}\circ \mathbf{p}$. In general, if $f=$ lim $\mathbf{f}$ is an epimorphism in $\mathcal{C}$, $\mathbf{f}$ need not be an epimorphism in pro-$\mathcal{C}$ (see \cite{Mard-Segal2},Ch. II, \S 6.1). But if $\mathbf{Y}\in $ pro-$\mathcal{C}$ is uniformly movable, then the statement is true (see \cite{Mard-Segal2}, Ch. II, \S 6.1, Th.3, and \cite {Moszy}). The following theorem is a generalization of this result.
\begin{thm}\label{thm.6.1}
Let $\mathbf{f}=[(f_\mu,\phi)]:\mathbf{X}=(X_\lambda,p_{\lambda\lambda'},\Lambda)\rightarrow \mathbf{Y}=(Y_\mu,q_{\mu\mu'},M)$ be a morphism in pro-$\mathcal{C}$, let $\mathbf{p}=[(p_\lambda)]:X \rightarrow \mathbf{X}$, $\mathbf{q}=[(q_\mu)]:Y\rightarrow \mathbf{Y}$ be inverse limit in $\mathcal{C}$ and let $f=$lim  $\mathbf{f}:X\rightarrow Y$. If $f$ is an epimorphism in $\mathcal{C}$, $\mathbf{f}$ is uniformly movable and for each index $\mu\in M$ the composition $f_\mu\circ p_{\phi(\mu)}$ is an epimorphism in $\mathcal{C}$, then $\mathbf{f}$ is an epimorphism in pro-$\mathcal{C}$.

\end{thm}
\begin{proof}
Let $\mathbf{g}=[(g_\nu,\psi)],\mathbf{g}'=[(g'_\nu,\psi')]:\mathbf{Y}\rightarrow \mathbf{Z}=(Z_\nu,r_{\nu\nu'},N)$ be morphisms in pro-$\mathcal{C}$ such that  $\mathbf{g}\circ \mathbf{f}=\mathbf{g}'\circ \mathbf{f}$. We must prove that $\mathbf{g}=\mathbf{g}'$. Since $\mathbf{f}\circ \mathbf{p}=\mathbf{q}\circ f$, it follows that $\mathbf{g}\circ \mathbf{q}\circ f=\mathbf{g}'\circ \mathbf{q}\circ f$, i.e., for every $\nu\in N$ one has
\begin{equation}
g_\nu\circ q_{\psi(\nu)}\circ f=g'_\nu\circ q_{\psi'(\nu)}\circ f
\end{equation}
and therefore
\begin{equation}
g_\nu\circ q_{\psi(\nu)}=g'_{\nu}\circ q_{\psi'(\nu)}.
\end{equation}
Choose $\mu\in M,\mu\geq \psi(\nu),\psi'(\nu)$, and let $\lambda\in \Lambda$, $\lambda\geq \phi(\mu)$ be a uniform movability index of $\mu$ with respect to $(f_\mu,\phi)$. Then there is a morphism $\mathbf{u}:X_{\lambda'}\rightarrow \mathbf{Y}$ in pro-$\mathcal{C}$ such that
\begin{equation}
f_{\mu\lambda}=\mathbf{q}_\mu\circ \mathbf{u}.
\end{equation}
Putting $u=$lim $ \mathbf{u}:X_{\lambda'}\rightarrow Y$, so that $\mathbf{q}\circ u=\mathbf{u}$,  we obtain
\begin{equation}
f_{\mu\lambda}=q_\mu\circ u.
\end{equation}
Consequently,
\begin{equation}
g_\nu\circ q_{\psi(\nu)\mu}\circ f_{\mu\lambda}=g_\nu\circ q_{\psi(\nu)\mu}\circ f_{\mu\lambda}=g_\nu\circ q_{\psi(\nu)}\circ q_\mu\circ u=g_\nu \circ q_{\psi(\nu)}\circ u.
\end{equation}
Similarly,
\begin{equation}
g'_\nu\circ q_{\psi'(\nu)\mu}\circ f_{\mu}\circ p_{\phi(\mu)\lambda}=g'_\nu\circ q_{\psi'(\nu)}\circ u,
\end{equation}
so that (6.2) implies
\begin{equation}
g_\nu\circ q_{\psi(\nu)\mu}\circ f_{\mu\lambda}=g'_\nu\circ q_{\psi'(\nu)\mu}\circ f_{\mu\lambda},
\end{equation}
and composing at right by $p_{\lambda}$ we have
\begin{equation}
g_\nu\circ q_{\psi(\nu)\mu}\circ (f_{\mu}\circ p_{\phi(\mu)})=g'_\nu\circ q_{\psi'(\nu)\mu}\circ (f_{\mu}\circ p_{\phi(\mu)}).
\end{equation}
From this, by the hypothesis this implies
\begin{equation}
g_\nu\circ q_{\psi(\nu)\mu}=g'_\nu\circ q_{\psi'(\nu)\mu},
\end{equation}
which shows that $(g_\nu,\psi)\sim (g'_\nu,\psi')$ and therefore $\mathbf{g}=\mathbf{g}'$ in pro-$\mathcal{C}$.
\end{proof}
\begin{rem}\label{rem.6.2}
The additional condition from Theorem \ref{thm.6.1}, that $f_{\mu}\circ p_{\phi(\mu)}$ is an epimorphism, depends only on the pro-morphism $\mathbf{f}$. Indeed, if $(f'_\mu,\phi')\sim (f_\mu,\phi)$, then for every $\mu\in M$ there exists $\lambda\geq \phi(\mu),\phi'(\mu)$ such that $f_\mu\circ p_{\phi(\mu)
\lambda}=f'_\mu\circ p_{\phi'(\mu)\lambda}$. From this, by composing at right with $p_\lambda$, we obtain $f_\mu\circ p_{\phi(\mu)}=f'_\mu\circ p_{\phi'(\mu)}$.
\end{rem}

We recall that if $\mathcal{C}$ is a category with zero-objects, then a morphism $f:A\rightarrow B$ of $C$ is a \textsl{weak} \textsl{epimorphism} if $u\circ f=0$ implies $u=0$. Weakened versions of the categorical notions of monomorphism and epimorphism have proved to be of some interest in pointed homotopy theory. In 1986, J.Roitberg \cite{Roitberg} used the properties of a remarkable group discovered by G. Higman \cite{Higman} to exhibit examples in the pointed homotopy category of weak epimorphisms which are not epimorphisms. Obviously that pro-Grp and pro-HTop$^\ast$ are categories with zero objects. Generally, if a category $\mathcal{C}$ is with zero-objects  the pro-category pro-$\mathcal{C}$ is with zero-objects.
Now the proof of Theorem \ref{thm.6.1} is perfectly adaptable to prove the following weakened version.
\begin{thm}\label{thm.6.3}
With the notations of Theorem \ref{thm.6.1} and supposing that the category $\mathcal{C}$ is with zero-objects, $f$ is a weak epimorphism in $\mathcal{C}$ and $\mathbf{f}$ is uniformly movable such that for each index $\mu\in M$ the composition $f_\mu\circ p_{\phi(\mu)}$ is a weak epimorphism in $\mathcal{C}$, then $\mathbf{f}$ is a weak  epimorphism in pro-$\mathcal{C}$.
\end{thm}
\begin{rem}\label{rem.6.4}
If $\mathbf{Y} $ is movable, then we obtain a weakened version of Moszynska' Theorem \cite{Moszy} (see also \cite{Mard-Segal2},Ch. II, \S 6.1, Th.3).
\end{rem}
\begin{cor}\label{cor.6.5}
Let $\mathbf{f}=[(f_\mu)]:X\rightarrow \mathbf{Y}=(Y_\mu,q_{\mu\mu'},M)$ be a morphism in pro-$\mathcal{C}$ and $f=lim \mathbf{f}:X\rightarrow Y=lim \mathbf{Y}$. If $f$ and every $f_\mu:X\rightarrow Y_\mu, \mu\in M $ are epimorphisms (weak epimorphisms)in $\mathcal{C}$, then $\mathbf{f}$ is an epimorphism(a weak epimorphism)in pro-$\mathcal{C}$.

Particularly, if $\mathbf{p}=(p_\lambda):X\rightarrow \mathbf{X}=(X_\lambda,p_{\lambda\lambda'},\Lambda)$ is an inverse limit with all projections $p_\lambda$ being epimorphisms(weak epimorphisms)  in $\mathcal{C}$, then $\mathbf{p}$ is an epimorphism (weak epimorphism)  in pro-$\mathcal{C}$.

\end{cor}

Now we can state two variants of  Dydak's infinite-dimensional Whitehead theorem in shape theory \cite{Dydak}.
\begin{cor}\label{cor.6.6}
Let $F:(X,\ast)\rightarrow (Y,\ast)$ be a weak shape equivalence of pointed connected topological spaces. Suppose that $X$ is of finite shape dimension, $sh X<\infty$, and that $F$ is a movable pointed shape domination. Then $F$ is a pointed shape equivalence.
\end{cor}
\begin{proof}
If $F$ is a pointed shape domination this  means that there exists a shape morphism $G:(Y,\ast)\rightarrow (X,\ast)$ such that $FG=1_{(Y,\ast)}$. Then we can apply Corollary \ref{cor.4.7} a) and we deduce that $(Y,\ast)$ is movable. Now the conclusion of corollary follows from \cite{Mard-Segal2}, Ch. II, \S 7.2, Theorem 6 (i).
\end{proof}
\begin{cor}\label{cor.6.7}
Let $F:(X,\ast)\rightarrow (Y,\ast)$ be a weak shape equivalence of pointed connected topological space. Suppose that $Y$ is of finite shape dimension, $sh Y<\infty$, and that $F$ is a co-movable pointed shape morphism which has a left inverse. Then $F$ is a pointed shape equivalence.
\end{cor}
\begin{proof}
By Corollary \ref{cor.5.16} in the shape version, $(X,\ast)$ is movable. Then the conclusion of this corollary follows from \cite{Mard-Segal2}, Ch. II, \S 7.2, Theorem 6 (ii).
\end{proof}

P.S.Gevorgyan\hfill I. Pop

Faculty of Mathematics\hfill Faculty of
Mathematics

Moscow State Pedagogical University\hfill "Al.I.Cuza" University

88, Vernadskogo,Moscow, Russia \hfill 700505 Iasi, Romania

E-mail: pgev@yandex.ru \hfill ioanpop@uaic.ro

\end{document}